\documentclass[titlepage,11pt]{article}
\usepackage{latexsym}
\usepackage{amsfonts}
\usepackage{amsmath}
\usepackage{amsthm}
\usepackage{tikz}
\usepackage{float}
\usepackage{lmodern}
\usepackage{enumerate}
\usepackage{comment}

\usepackage{hyperref}

\usepackage{cleveref}

\usepackage{xcolor}
\usepackage[T1]{fontenc}

\def\dd{\hbox{-}}

\def\cc{\hbox{-}\cdots\hbox{-}}

\def\ll{,\ldots,}
\def\cupcup{\cup\cdots\cup}

\usepackage{setspace}

\title{Detecting a long even hole}
\author{Linda Cook\\
Princeton University, Princeton, NJ 08544, USA
\\
\\
Paul Seymour\thanks{Partially supported by AFOSR grant A9550-19-1-0187 and NSF grant DMS-1800053.}\\
Princeton University, Princeton, NJ 08544, USA}

\date{August 28, 2019; revised \today}
\newtheorem{theorem}{Theorem}[section]

\newtheorem{lem}[theorem]{Lemma}

\begin{document}
\maketitle

\begin{abstract}
For each integer $\ell \geq 4$, we give a polynomial-time algorithm to test whether a graph contains an induced cycle with length at least $\ell$ and even.
\end{abstract}

\section{Introduction}
All graphs in this paper are finite and have no loops or parallel edges. A {\em hole} in a graph is an induced subgraph which is a cycle of length
at least four. The {\em length} of a path or cycle $A$ is the number of edges in $A$ and the {\em parity} of $A$ is the parity of its length.
For graphs $G, H$ we will say that $G$ {\em contains} $H$ if some induced subgraph of $G$ is isomorphic to $H$. 
We say $G$ is {\em even-hole-free} if $G$ does not contain an even hole. 
We denote by $|G|$ the number of vertices of a graph $G$.
A graph algorithm is {\em polynomial-time} if its running time is at most polynomial in $|G|$.

This paper concerns detecting holes in a graph with length satisfying certain conditions. It is of course trivial
to test for the existence of a  
hole of length at least $\ell$, in polynomial time for each
constant $\ell\ge 4$, as follows. We enumerate all induced paths $P$ of length $\ell-2$. For each choice of $P$, let its ends 
be $x$ and $y$, let $P^*= V(P) \setminus \{x,y\}$, and let $N$ be the set of vertices different from $x,y$ that belong 
to or have a neighbour in $P^*$.
Then we check whether $x$ 
and $y$ are in the same component of $G \setminus N$. This depends on $\ell$ being fixed; if $\ell$ is part of the input, then the problem is NP-complete, because it contains the hamilton cycle problem.

But the problem is much less trivial if we impose restrictions on the parity of the hole length, or more generally on 
its residue class modulo some fixed number.
Sepehr Hajebi~\cite{w1hard} provided a proof in private communication that if $\ell$ is part of the input, then detecting holes 
of length at least $\ell$ in a specific residue class is W[1]-hard, and thus not
fixed-parameter tractable unless the central conjecture of parameterized complexity theory is false. More exactly, for all
integers $m,r$ with $m\ge 2$ and $0\le r<m$,
if there is
an algorithm that, with input $G, \ell$, determines  in time $\mathcal{O}(f(\ell)p(|G|))$ whether $G$ contains a hole $C$ of length at length at least $\ell$ and
with $|E(C)| \cong r \mod m$, where $f$ is some computable function and $p$ is a polynomial, then the central conjecture of 
parameterized complexity theory would be false. But this is different from what we are doing in this paper: we are working with 
$\ell$ fixed, and Hajebi wants $\ell$ part of the input.

Here is an overview of positive results about algorithms to detect even and odd holes, with odd holes first:
\begin{itemize}
\item In 2005, Chudnovsky, Cornu\'{e}jols, Liu, Seymour and Vu\v{s}kovi\'{c}~\cite{bergetest} gave an $\mathcal{O}(|G|^9)$ algorithm 
to test whether a graph $G$ or its complement has an odd hole.
\item In 2019, Chudnovsky, Scott, Seymour, and Spirkl~\cite{chudnovsky2020oddhole} gave an algorithm to detect an odd hole in $G$ in time 
$\mathcal{O}(|G|^9)$; and Lai, Lu and Thorup~\cite{lai2019threeinatree} improved this running time to $\mathcal{O}(|G|^8)$.
\item Also in 2019, Chudnovsky, Scott and Seymour~\cite{chudnovsky2019detectinglongodd} gave a $\mathcal{O}(|G|^{20\ell + 40})$ algorithm to test whether $G$ contains an
odd hole of length at least $\ell$, where $\ell$ is any fixed number.
\item In 2020, Chudnovsky, Scott and Seymour~\cite{chudnovsky2020shortestodd} gave an $\mathcal{O}(|G|^{14})$ algorithm that 
finds a shortest odd hole in $G$ (if there is one) in time $\mathcal{O}(|G|^{14})$.
\end{itemize}

For even holes the story is a little different:
\begin{itemize}
    \item In 2002, Conforti, Cornu\'{e}jols, Kapoor and Vu\v{s}kov\'{i}c~\cite{conforti2002even} gave an approximately 
$\mathcal{O}(|G|^{40})$ algorithm to test whether a graph contains an even hole, by using a structure theorem about 
even-hole-free graphs from an earlier paper~\cite{conforti2002evenstructure}.
    \item In 2003, Chudnovsky, Kawarabayashi, and Seymour~\cite{chudnovsky2005detectingeven} provided a simpler algorithm that searches for even holes directly in $\mathcal{O}(|G|^{31})$.
    \item In 2015, Chang and Lu~\cite{chang} gave an $\mathcal{O}(|G|^{11})$ algorithm to determine whether a graph contains an even hole; and 
Lai, Lu and Thorup~\cite{lai2019threeinatree} improved this running time to $\mathcal{O}(|G|^9)$ in 2020.
\end{itemize}
So, some analogues of what has been done for odd holes are still open for even holes. 
The problem of detecting long even holes is one of them, and is what is solved in this paper. Another is finding a shortest even hole
if there is one; this we have not yet been able to solve.

We remark that even versus odd has been almost the entire focus of previous research, but what about holes of length 
a multiple of three, can we detect them in polynomial time? Triangle-free graphs with no holes of length a multiple of three
have some very interesting properties~\cite{ternary}, but we currently have no idea how to recognize such graphs.

Since we are looking for even holes of length at least $\ell$, we might as well assume that $\ell$ is even.
Our main result is the following:

\begin{theorem}\label{mainthm}
For each even integer $\ell \geq 4$, there is an algorithm with the following specifications:
\begin{description}
\item [Input:] A graph $G$.
\item [Output:] Decides whether $G$ has an even hole of length at least $\ell$.
\item [Running time:] $\mathcal{O}(|G|^{9\ell+3})$.
\end{description}
\end{theorem}

Our algorithm combines approaches described in \cite{chudnovsky2005detectingeven} and \cite{chudnovsky2019detectinglongodd}. 
The algorithm uses a technique called ``cleaning'', as do the algorithms of \cite{chudnovsky2005detectingeven},
\cite{chudnovsky2019detectinglongodd} and many other algorithms to detect induced subgraphs.

Here is an outline of the method. The result is clear if $\ell=4$, so we might as well assume that $\ell\ge 6$; and throughout 
the paper $\ell \geq 6$ is a fixed even integer, and a {\em long} hole or path is a hole or 
path of length at least $\ell$.
A {\em shortest} long even hole is a long even hole of minimum length.
If $C$ is a hole in $G$, a vertex $v$ of $V(G) \setminus V(C)$ is {\em $C$-major} if there is no three-vertex path of $C$
containing all neighbours of $v$ in $V(C)$. A hole $C$ is {\em clean} if it has no $C$-major vertex.

\begin{itemize}
\item First, we test for the presence in the input graph $G$ of certain kinds of induced subgraphs (``short'' long even holes, 
``long jewels of bounded order'', ``long thetas'', `` long ban-the-bombs'', ``long near-prisms'') that are detectable in 
polynomial time (sometimes, under the assumption that earlier graphs in this list are not present) and whose presence would 
imply that $G$ contains a long even hole. We call these kinds of subgraphs ``easily-detectable configurations.'' We may assume these 
tests are unsuccessful.

\item Second, we generate a {\em cleaning list}, a list of polynomially many subsets of $V(G)$ such that if $C$ is a shortest long even 
hole in $G$, then for some set $X$ in the list, $X$ contains every $C$-major vertex and no vertex of $C$. This process depends on 
the absence of easily-detectable configurations.
\item
Third, for every $X$ in our cleaning list we check whether $G \setminus X$ contains a clean shortest long even hole. This also 
depends 
on the absence of easily-detectable configurations.
We detect a clean shortest long even hole $C$ by guessing three evenly-spaced vertices along $C$ and taking shortest paths 
between them.

\end{itemize}

We are calling long near-prisms easily-detectable configurations, 
but ``easily'' might be a misnomer, because this is by far the 
computationally most expensive step of the algorithm, and the bulk of what is novel in the paper. 
For a general graph $G$, deciding whether $G$ 
contains a long near-prism 
is NP-complete;  Maffray and Trotignon's proof \cite{maffray2005algorithms} that deciding whether $G$ contains a prism is 
NP-complete can easily be adjusted to prove that deciding whether $G$ contains a long near-prism is NP-complete. Fortunately
it really is easy to detect the other ``easily detectable'' configurations, so we can assume there are none; and in such graphs we can detect the presence of long near-prisms in polynomial time.

The approach of determining whether $G$ contains an even hole by first testing whether $G$ contains a theta of a prism was outlined 
in \cite{chudnovsky2005detectingeven}. Moreover, Chudnovsky and Kapadia gave an algorithm to decide whether $G$ contains a theta or 
a prism in \cite{chudnovsky2008thetaprism}. Their algorithm does not translate directly to long theta and long near-prism detection, but 
we were able to use a similar algorithmic structure for our purposes.

\section{The easily-detectable configurations}
The {\em interior} $P^*$ of a path $P$ is the set of vertices of $P$ that are not ends of $P$.
Thus $P^* = \emptyset$ for a path $P$ of length at most one. If $X, Y \subseteq V(G)$, we say $X$ is 
{\em anticomplete} to $Y$ if $X\cap Y=\emptyset$ and no vertex in $X$ is adjacent to a vertex in $Y$. 
We begin with a test for what we called ``short'' long even holes:
\begin{theorem} \label{alg:shortlongevenholes}
For each integer $k \geq \ell$, there is an algorithm with the following specifications:
\begin{description}
\item[Input:] A graph $G$.
\item[Output:] Decides whether $G$ has a long even hole of length at most $k$.
\item[Running time:] $\mathcal{O}(|G|^k)$.
\end{description}
\end{theorem}

\begin{proof}
We enumerate all vertex sets of size $\ell, \ell+1, \dots, k$ and for each one, check whether it induces an even hole.
\end{proof}

We need the following easily-detectable configuration of \cite{chudnovsky2019detectinglongodd} (slightly modified).
Let $u,v \in V(G)$ and let $Q_1, Q_2$ be induced paths between $u,v$ of different parity. Let $P$ be an induced path between 
$u,v$ of length at least $\ell-\min(|E(Q_1)|,|E(Q_2)|)$, such that $P^*$ is anticomplete to $Q_1^*\cup Q_2^*$. We say the subgraph induced 
on $V(P \cup Q_1 \cup Q_2)$ is a {\em long jewel of order $\max{(|V(Q_1)|, |V(Q_2)|)}$ formed by $Q_1, Q_2, P$}. Any graph 
containing a long jewel has a long even hole, since the holes $P\cup Q_1$ and $P \cup Q_2$ are both long holes and one of them is even.

We need a slight extension of Theorem 2.2 of \cite{chudnovsky2019detectinglongodd}:

\begin{theorem} \label{alg:longjewels}
There is an algorithm with the following specifications.
\begin{description}
\item[Input:] A graph $G$ and an integer $k \geq 0$.
\item[Output:] Decides whether $G$ has a long jewel of order at most $k$.
\item[Running time:] $\mathcal{O}(|G|^{n})$ where $n=k+1+\max(k,\ell-1)$.
\end{description}
\end{theorem}

\begin{proof}
We enumerate all triples of induced paths $Q_1,Q_2,R$ of $G$, such that:
\begin{itemize}
\item $Q_1,Q_2$ join the same pair of vertices, say $u,v$;
\item one of $Q_1,Q_2$ is odd and the other is even, and each has at most $k$ vertices;
\item $R$ has length $\ell-\min(|E(Q_1)|,|E(Q_2)|)-2$ (or zero if this number is negative), 
and has one end $u$ and the other some vertex $w$ say;
\item no vertex of $V(R)\setminus \{u\}$ equals or has a neighbour in $V(Q_1\cup Q_2)\setminus \{u\}$.
\end{itemize}
For each such triple of paths, let $X$ be the set of vertices of $G$ that are different from and nonadjacent to each vertex of
$V(Q_1\cup Q_2\cup R)\setminus \{v,w\}$. We test whether there is a path in $G[X\cup\{w,v\}]$ between $w,v$. If so we output that $G$ contains
a long jewel of order at most $k$. If no triple yields this outcome, we output that $G$ has no such long jewel.

To see the correctness of the algorithm, certainly the output is correct if $G$ contains no long jewel of order at most $k$. Suppose then it does, say
formed by $Q_1,Q_2,P$. Let $u,v$ be the ends of $P$, and let $R$ be the subpath of $P$ of length  $\ell-\min(|E(Q_1)|,|E(Q_2)|)-2$ (or zero if this number is negative) with one end $u$. When the algorithm
tests the triple $Q_1,Q_2,R$, it will discover there is a path in $G[X\cup \{w,v\}]$ between $w,v$, because the remainder of $P$ is such a path. Consequently
the output is correct.

The running time is $O(|G|^2)$ for each triple of paths, and there are at most $|G|^{n}$ such triples where $n = k-1+\max(k,\ell-1)$, 
so the running time is as claimed.
This proves Theorem ~\ref{alg:longjewels}.
\end{proof}

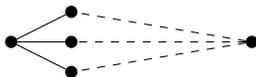
\begin{figure}[H]
\centering

\begin{tikzpicture}[scale=.8,auto=left]
\tikzstyle{every node}=[inner sep=1.5pt, fill=black,circle,draw]
\node (z) at (0,0) {};
\node (a) at (1,.5) {};
\node (b) at (1,0) {};
\node (c) at (1,-.5) {};
\node (d) at (4,0) {};

\foreach \from/\to in {z/a,z/b, z/c}
\draw [-] (\from) -- (\to);
\foreach \from/\to in {d/a,d/b, d/c}
\draw [dashed] (\from) -- (\to);

\end{tikzpicture}

\caption{A theta (dashed lines mean paths of arbitrary positive length)} \label{theta}
\end{figure}

A {\em theta} is a graph consisting of two non-adjacent vertices $u,v$ and three paths $P_1, P_2, P_3$ joining $u,v$ with pairwise 
disjoint interiors, and we say $P_1, P_2, P_3$ {\em form} a theta. The union of any two of $P_1,P_2,P_3$ is a hole, and
a {\em long theta} is a theta where all three holes are long.
If $G$ contains a long theta, then it contains a long even hole, because 
at least two of $P_1, P_2, P_3$ must have the same parity. To detect long thetas, 
we use the ``three-in-a-tree'' algorithm of~\cite{chudnovsky2010three}, the following:

\begin{theorem} \label{alg:threeinatree}
There is an algorithm with the following specifications:
\begin{description}
\item[Input:] A graph $G$ and three vertices $v_1,v_2,v_3$ of $G$.
\item[Output:] Decides whether there is an induced subgraph $T$ of $G$ with $v_1,v_2,v_3 \in V(T)$ such that $T$ is a tree.
\item[Running time:] $\mathcal{O}(|G|^4)$.
\end{description} 
\end{theorem}

Chudnovsky and Seymour's algorithm in \cite{chudnovsky2010three} to detect a theta in a graph $G$ can be adjusted to detect a long theta, as follows:

\begin{theorem}\label{alg:longtheta}
There is an algorithm with the following specifications:
\begin{description}
\item[Input:] A graph $G$.
\item[Output:] Decides whether $G$ contains a long theta.
\item[Running time:] $\mathcal{O}(|G|^{2\ell-1})$.
\end{description}
\end{theorem}
\begin{proof}
The algorithm is as follows. Say (temporarily) a {\em claw} is a graph that is the union of three paths $Q_1,Q_2,Q_3$, with
a common end $a$ and otherwise vertex-disjoint, of lengths $k_1,k_2,k_3$ respectively where $k_1,k_2,k_3\ge 2$, and 
$k_1+k_2, k_2+k_3, k_3+k_1\ge \ell-2$, and $k_1+k_2+k_3\le 2\ell-6$. If three paths $P_1,P_2,P_3$ of $G$ form a long theta, 
then $P_1\cup P_2\cup P_3$ includes 
a claw which is an induced subgraph of $G$. (To see this, if $P_1,P_2,P_3$ all have length at least $\ell/2$ take $Q_1,Q_2,Q_3$
all of length $\ell/2-1$, and if say $P_3$ has length less than $\ell/2$, take $Q_3=P_3$ and $Q_1,Q_2$ of length $\ell-2-|E(P_3)|$.) 
Conversely, if three paths form a theta that includes a claw, then the theta is long.

Let $B$ be a claw in $G$, and let $q_1,q_2,q_3$ be its three vertices of degree one in $B$. Let
$G'$ be the graph obtained from $G$ by deleting all vertices different from $q_1,q_2,q_3$ that belong to or have a neighbour in 
$V(B)\setminus \{q_1,q_2,q_3\}$. Then 
$B$ is an induced subgraph of a theta (and hence of a long theta) in $G$ if and only if there is an induced tree $T$ containing $q_1,q_2,q_3$
in $G'$.

So the algorithm is: enumerate all induced claws, and for each one, check if there is an induced tree as above. Since claws have 
at most $2\ell-5$ vertices, there are only $\mathcal{O}(|G|^{2\ell-5})$ of them, so the running time is $\mathcal{O}(|G|^{2\ell-1})$.
This proves Theorem \ref{alg:longtheta}.
\end{proof}

Lai, Lu and Thorup~\cite{lai2019threeinatree} provide a faster algorithm for the three-in-a-tree problem. Using their 
$\mathcal{O}(|E(G)|(\log |G|)^2)$ algorithm we can reduce the running time for detecting a long theta to 
$\mathcal{O}(|G|^{2\ell-3}(\log|G|)^2)$, but this improvement does not affect the asymptotic running time of our long even holes 
detection algorithm.

For brevity, it is convenient to describe enumerating all subgraphs of a certain type as ``guessing'' subgraphs of that 
type. In this language the algorithm can be written as follows: We guess the paths $Q_1$, $Q_2$ and $Q_3$ 
and test whether $q_1, q_2, q_3$ are contained in some induced tree of $G'$.

We call a path $P$ with ends $x, y$ an {\em xy-path}.
If $P$ is a path, and $x,y\in V(P)$, we denote the subpath of $P$ with ends $x,y$ by $x\dd P\dd y$. A path with vertices $v_1\ll v_k$
in order is denoted by $v_1\cc v_k$. If $P,Q$ are paths with ends $u,v$ and $v,w$ respectively, and their union is a path with ends $u,w$, we denote
this path by $u\dd P\dd v\dd Q\dd w$; and extend this notation for longer concatenations similarly.

Let us say a {\em ban-the-bomb} is a graph consisting of 
\begin{itemize}
\item a cycle $u\dd v_1\dd w\dd v_2\dd u$ of length four, and possibly 
the edge $uw$ (but we insist that $v_1,v_2$ are nonadjacent); and one further vertex $x$ adjacent to $u$, and nonadjacent to $v_1,v_2,w$; and
\item for $i = 1,2$, an $xv_i$-path $P_i$ of length at least two, where $P_i^*$ is anticomplete to $\{u,v,w\}$, and
$V(P_1)\setminus \{x\}$ is anticomplete
to $V(P_2)\setminus \{x\}$.
\end{itemize}
Thus it has three holes; and it is {\em long} 
if all three holes are long.
It is easy to see 
that every graph containing a long ban-the-bomb has a long even hole.

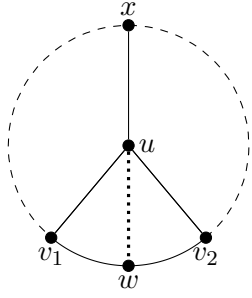
\begin{figure}[H]
\centering

\begin{tikzpicture}[scale=0.8,auto=left]
\tikzstyle{every node}=[inner sep=1.5pt, fill=black,circle,draw]

\def\r{2}
\def\a{40}
\node (u) at (0,0) {};
\draw[domain=-(90-\a):(270-\a),smooth,variable=\x,dashed] plot ({\r*cos(\x)},{\r*sin(\x)});
\draw[domain=(270-\a):(270+\a),smooth,variable=\x] plot ({\r*cos(\x)},{\r*sin(\x)});
\node (v1) at ({\r*cos(230)},{\r*sin(230)}) {};
\node (v2) at ({\r*cos(310)},{\r*sin(310)}) {};
\node (w) at ({\r*cos(270)},{\r*sin(270)}) {};
\node (x) at ({\r*cos(90)},{\r*sin(90)}) {};
\draw (u)--(v1);
\draw (u)--(v2);
\draw (u)--(x);
\draw[dotted, very thick] (u)--(w);
\draw (v1)--(u)--(v2);

\tikzstyle{every node}=[]
\draw (v1) node [below]           {$v_1$};
\draw (v2) node [below]           {$v_2$};
\draw (w) node [below]           {$w$};
\draw (u) node [right]           {$u$};
\draw (x) node [above]           {$x$};

\end{tikzpicture}
\caption{A ban-the-bomb. The dotted line is a possible edge.} \label{banthebomb}
\end{figure}

If there is no long theta, we can also search for long ban-the-bombs using the three-in-a-tree algorithm, as follows.
\begin{theorem}\label{alg:longbanthebomb}
There is an algorithm with the following specifications:
\begin{description}
\item[Input:] A graph $G$ with no long theta.
\item[Output:] Decides whether $G$ contains a long ban-the-bomb.
\item[Running time:] $\mathcal{O}(|G|^{2\ell+1})$.
\end{description}
\end{theorem}
\begin{proof}
Let us say a {\em bomb} is a graph consisting of a path $R$ of length $2\ell-6$, with middle three vertices $v_1\dd w\dd v_2$ in order
and two more vertices $u,v$, where $u$ is adjacent to $v_1,v_2$ and possibly to $x$, but to no other vertices of $R$,
and $x$ adjacent to $u$ but to no vertex of $R$.
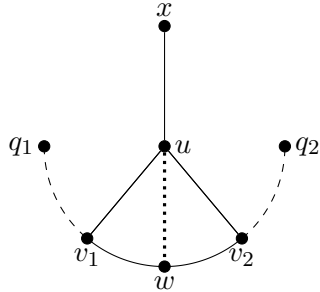
\begin{figure}[H]
\centering

\begin{tikzpicture}[scale=0.8,auto=left]
\tikzstyle{every node}=[inner sep=1.5pt, fill=black,circle,draw]

\def\r{2}
\def\a{40}
\node (u) at (0,0) {};
\draw[domain=-(90-\a):0,smooth,variable=\x,dashed] plot ({\r*cos(\x)},{\r*sin(\x)});
\draw[domain= 180:(270-\a),smooth,variable=\x,dashed] plot ({\r*cos(\x)},{\r*sin(\x)});
\draw[domain=(270-\a):(270+\a),smooth,variable=\x] plot ({\r*cos(\x)},{\r*sin(\x)});
\node (v1) at ({\r*cos(230)},{\r*sin(230)}) {};
\node (v2) at ({\r*cos(310)},{\r*sin(310)}) {};
\node (w) at ({\r*cos(270)},{\r*sin(270)}) {};
\node (x) at ({\r*cos(90)},{\r*sin(90)}) {};
\node (q1) at ({\r*cos(180)},{\r*sin(180)}) {};
\node (q2) at ({\r*cos(0)},{\r*sin(0)}) {};

\draw (u)--(v1);
\draw (u)--(v2);
\draw (u)--(x);
\draw[dotted, very thick] (u)--(w);
\draw (v1)--(u)--(v2);

\tikzstyle{every node}=[]
\draw (v1) node [below]           {$v_1$};
\draw (v2) node [below]           {$v_2$};
\draw (w) node [below]           {$w$};
\draw (u) node [right]           {$u$};
\draw (x) node [above]           {$x$};
\draw (q1) node [left]           {$q_1$};
\draw (q2) node [right]           {$q_2$};

\end{tikzpicture}
\caption{A bomb. The dashed lines are paths of length $2\ell-6$, and the dotted line is a possible edge.} \label{bomb}
\end{figure}

If there is a long ban-the-bomb in $G$, with vertices $u,v_1,v_2,w,x$ and paths $P_1,P_2$ as in the definition, then $P_1,P_2$
both have length at least $\ell-2$. For $i = 1,2$
let $Q_i$ be the subpath of $P_i$ of length $\ell-4$ with one end $v_i$, and let $q_i$ be the other end of $Q_i$; then
the subgraph induced on $V(Q_1\cup Q_2)\cup \{u,w,x\}$ is a bomb.
To search for long ban-the-bombs, we enumerate all induced subgraphs of $G$ that are bombs. For each such induced bomb $B$, let $L$
be its three vertices of degree one; check if there is an induced tree containing the vertices in $L$, in the graph obtained from $G$
by deleting all vertices not in $L$ that belong to or have neighbours in $V(B)\setminus L$.  If so, output that $G$ contains
a long ban-the-bomb and stop. If no bomb has such a tree, output that there is no long ban-the-bomb.

This concludes the description of the algorithm. A bomb has $2\ell-3$ vertices, so there are $\mathcal{O}(|G|^{2\ell-3})$ choices for the 
bomb, and the running time is
$\mathcal{O}(|G|^{2\ell+1})$.

If a bomb is contained in a long ban-the-bomb, then such a tree exists, and the outcome is correct, but the converse is less clear.
Suppose that for some bomb $B$ there is a tree $T$ as described in the algorithm. Let $R$ be as in the definition of a bomb,
with ends $q_1,q_2$ where $q_1,v_1,w,v_2,q_2$ are in order. 
There is a vertex
$t\in V(T)$ and three paths $T_1,T_2,T_3$ of $T$ (possibly of length zero) between $t$ and $q_1,q_2,x$ respectively,
pairwise anticomplete except for $t$.
For $i = 1,2$, the hole $t\dd T_i\dd q_i\dd R\dd v_1\dd w\dd x\dd T_3\dd t$ is long for $i = 1,2$, since $T_i\cup T_3$ has length at least two;
so if 
$t\ne x$, there is a long theta formed by the paths $t\dd T_i\dd q_i\dd R\dd v_i\dd u$ for $i = 1,2$, and the path
$t\dd T_3\dd x\dd u$, a contradiction. Thus $t=x$, and we have a long ban-the-bomb. This proves correctness.

\end{proof}

A {\em triangle} is a graph consisting of three pairwise adjacent vertices.
A {\em near-prism} is a graph consisting of two triangles with vertex sets $\{a_1, a_2, a_3\}$ and $\{b_1, b_2, b_3\}$, 
sharing at most one vertex, and three pairwise vertex-disjoint paths $P_1, P_2, P_3$, such that $P_i$ has ends $a_i$ and $b_i$ for 
$i \in \{1,2,3\}$, and it is {\em long} if the subgraph induced on $V(P_i\cup P_j)$ is a long hole for all distinct $i,j\in \{1,2,3\}$.
It is a {\em prism} if the two triangles are vertex-disjoint. 
We call $P_1, P_2, P_3$ the {\em constituent paths} of the near-prism. It is easy to see that every graph with a long near-prism has a 
long even hole. 

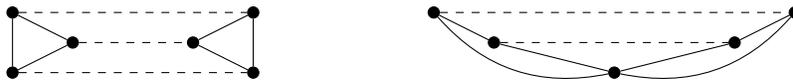
\begin{figure}[H]
\centering
\begin{tikzpicture}[scale=.8,auto=left]
\tikzstyle{every node}=[inner sep=1.5pt, fill=black,circle,draw]
\node (a1) at (1,0) {};
\node (a2) at (0,.5) {};
\node (a3) at (0,-.5) {};
\node (b1) at (3,0) {};
\node (b2) at (4,.5) {};
\node (b3) at (4,-.5) {};

\foreach \from/\to in {a1/a2,a1/a3,a2/a3,b1/b2,b1/b3,b2/b3}
\draw [-] (\from) -- (\to);
\foreach \from/\to in {a1/b1,a2/b2,a3/b3}
\draw [dashed] (\from) -- (\to);

\node (a1) at (7,.5) {};
\node (a2) at (8,0) {};
\node (a3) at (10,-.5) {};
\node (b1) at (13,.5) {};
\node (b2) at (12,0) {};

\foreach \from/\to in {a1/a2,a2/a3,b1/b2,b2/a3}
\draw [-] (\from) -- (\to);
\draw[bend left=30] (a3) to (a1);
\draw[bend right=30] (a3) to (b1);
\foreach \from/\to in {a1/b1,a2/b2}
\draw [dashed] (\from) -- (\to);

\end{tikzpicture}

\caption{Near-prisms.} \label{nearprism}
\end{figure}

\section{Detecting a clean lightest long near-prism}\label{sec:sketch}

Our next goal is a poly-time algorithm to test whether $G$ contains a long near-prism, for graphs $G$ that contain
none of the other easily-detectable configurations; and this and the next two sections are devoted to this.
Let us say a graph $G$ is a {\em prospect} if $G$ contains no long even hole of length at most $2\ell$, no long jewel of order at most $\ell+1$,
no long theta and no long ban-the-bomb.
We will show the following (the outline of this algorithm is like that of \cite{chudnovsky2008thetaprism}):
\begin{theorem}\label{alg:longprisms}
There is an algorithm with the following specifications:
\begin{description}
\item[Input:] A prospect $G$.
\item[Output:] Decides whether $G$ contains a long near-prism.
\item[Running time:] $\mathcal{O}(|G|^{9\ell+3})$.
\end{description}
\end{theorem}

We need: 
\begin{lem}\label{longerprism}
Let $G$ be a prospect, and let $K$ be a long near-prism in $G$, with constituent paths $P_1,P_2,P_3$.
Then at least two of $P_1,P_2,P_3$ have length at least $\ell$. 
\end{lem}

\begin{proof}

Suppose that $P_1,P_2$ both have length less than $\ell$. Then the long hole induced on $V(P_1\cup P_2)$ is odd, since its length is 
between $\ell$ and $2\ell$, and $G$ is a prospect; and so the paths $a_3\dd a_1\dd P_1\dd b_1\dd b_3$ and 
$a_3\dd a_2\dd P_2\dd b_2\dd b_3$
have different parity. Hence these two paths with $P_3$ form a long jewel of order at most $\ell+1$, a contradiction. This proves
Lemma \ref{longerprism}.

\end{proof}

For a graph $G$ and $x,y \in V(G)$, we call the length of a shortest $xy$-path in $G$ the {\em $G$-distance} between $x$ and $y$ and denote it by $d_G(x,y)$. 
Let us say a
{\em frame} $F$ is a graph 
with the following properties:
\begin{itemize}
\item $F$ is the union of two triangles with vertex sets $A,B$ with at most one vertex in common, 
and three graphs $F_1,F_2,F_3$ that are pairwise vertex-disjoint; and each of $F_1,F_2,F_3$ has exactly one vertex in $A$ and one in $B$;
\item for $1\le i\le 3$, $F_i$ is either a path with one end in $A$ and the other in $B$ of length at most $\ell-1$, or the disjoint union of 
two paths, both of length exactly $\ell/2-1$,
one with an end in $A$ and the other with an end in $B$ (it follows that if the triangles share a vertex then one of the $F_i$ is a path of length zero); and
\item at most one of $F_1,F_2,F_3$ is a path.
\end{itemize}
\begin{figure}[H]
\centering
\begin{tikzpicture}[scale=.8,auto=left]
\tikzstyle{every node}=[inner sep=1.5pt, fill=black,circle,draw]
\def\r{8}
\def\s{6}
\node (a1) at (.7,0) {};
\node (a2) at (0,.8) {};
\node (a3) at (0,-.8) {};
\node (b1) at (3.3,0) {};
\node (b2) at (4,.8) {};
\node (b3) at (4,-.8) {};
\node (s1) at (1.5,0) {};
\node (s2) at (0.8,.8) {};
\node (s3) at (0.8,-.8) {};
\node (t1) at (2.5,0) {};
\node (t2) at (3.2,.8) {};
\node (t3) at (3.2,-.8) {};

\foreach \from/\to in {a1/a2,a1/a3,a2/a3,b1/b2,b1/b3,b2/b3}
\draw [-] (\from) -- (\to);
\foreach \from/\to in {a1/s1, t1/b1,a2/s2,t2/b2,a3/s3,t3/b3}
\draw [dashed] (\from) -- (\to);

\node (a1) at (13.5-\r,.8) {};
\node (a2) at (14.5-\r,0) {};
\node (a3) at (16-\r,-.8) {};
\node (b1) at (18.5-\r,.8) {};
\node (b2) at (17.5-\r,0) {};
\node (s1) at (14.3-\r,.8) {};
\node (s2) at (15.3-\r,0) {};
\node (t1) at (17.7-\r,.8) {};
\node (t2) at (16.7-\r,0) {};

\foreach \from/\to in {a1/a2,a2/a3,b1/b2,b2/a3}
\draw [-] (\from) -- (\to);
\draw[bend left=30] (a3) to (a1);
\draw[bend right=30] (a3) to (b1);
\foreach \from/\to in {a1/s1, t1/b1, a2/s2, t2/b2}
\draw [dashed] (\from) -- (\to);

\node (a1) at (6.7+\s,0) {};
\node (a2) at (6+\s,.8) {};
\node (a3) at (6+\s,-.8) {};
\node (b1) at (9.3+\s,0) {};
\node (b2) at (10+\s,.8) {};
\node (b3) at (10+\s,-.8) {};
\node (s1) at (7.5+\s,0) {};
\node (s2) at (6.8+\s,.8) {};
\node (t1) at (8.5+\s,0) {};
\node (t2) at (9.2+\s,.8) {};

\foreach \from/\to in {a1/a2,a1/a3,a2/a3,b1/b2,b1/b3,b2/b3}
\draw [-] (\from) -- (\to);
\foreach \from/\to in {a1/s1, t1/b1,a2/s2,t2/b2,a3/b3}
\draw [dashed] (\from) -- (\to);

\end{tikzpicture}

\caption{Frames} \label{frames}
\end{figure}

We call $A,B$ the {\em bases} of the frame.
A {\em frame in $G$} means an induced subgraph
of $G$ that is a frame. The {\em ends} of a frame $F$ are its vertices of degree one,
and the set of vertices of $F$ that are not ends of $F$
is denoted by $F^*$. 


If $K$ is a near-prism in a prospect $G$, then $K$ is long if and only if $K$ contains a frame, by Lemma \ref{longerprism}.
(Indeed, every long near-prism contains a unique frame, which we denote by $F_K$.)
Thus if at some stage
we have a frame $F$ in the input prospect $G$, and we find a near-prism $K$ of $G$ containing $F$, then we know that $K$ 
is long without having to check the lengths of the missing parts of its constituent paths; and conversely, if there is a long near-prism $K$
in $G$, and we examine all frames in $G$ and test, for each one, whether it is contained in a long near-prism, then eventually we will 
test $F_K$ and report success.
(Enumerating all frames can be 
done in polynomial time, since frames have a bounded number of vertices; the more difficult issue is to handle a given frame in polynomial time.)

That is our basic method, to try all frames and see if they can be extended to long near-prisms. But it is helpful to have a little
more information about the long near-prism we are looking for than just its frame. For instance, for the first frame in figure \ref{frames}, we would like to know which vertex of the left triangle corresponds to which one of the right.
Let us say an {\em ordered frame} $\mathcal{F}$
consists of a frame $F$ together with a linear ordering of both of its bases.
Let $K$ be a long near-prism and let $F_K$ be its frame. Let $K$ have constituent paths $P_1,P_2,P_3$, where 
$|E(P_1)|\le |E(P_2)|\le |E(P_3)|$, and for $1\le i\le 3$ let $P_i$ have ends $a_i, b_i$. Then these six (or possibly, five)
vertices belong to the two bases of $F_K$, and we would like to know this labelling. 
We define $\mathcal{F}_K$ to be the ordered frame consisting of 
$F_K$ and the orderings $a_1<a_2<a_3$ and $b_1<b_2<b_3$, and call it an {\em ordered frame of $K$}. 
(It is not quite unique, because two of $P_1,P_2,P_3$ might have the same length.) 

If $\mathcal{F}$ is an ordered frame of a long near-prism
$K$, with bases $\{a_1,a_2,a_3\},\{b_1,b_2,b_3\}$ and constituent paths $P_1,P_2,P_3$, where $P_i$ has ends $a_i, b_i$ for $i = 1,2,3$,
we say that $P_1, P_2, P_3$ are {\em numbered according to $\mathcal{F}$} if the orderings of $\mathcal{F}$ are $a_1<a_2<a_3$ and
$b_1<b_2<b_3$.

If $K$ is a long near-prism in $G$, we call a vertex $q \in V(G) \setminus V(K)$ {\em $K$-major} if there is no three-vertex 
path of $K$ containing all neighbours
of $q$ in $V(K)$, and we say $K$ is {\em clean} if there are no $K$-major vertices.
We call a long near-prism $K'$ {\em shorter}
than a long near-prism $K$ if $|V(K')| < |V(K)|$, and thereby define a {\em shortest} long near-prism.

We will test for long near-prisms as follows. Shortest long near-prisms have special properties that make them easier to
detect than general long near-prisms, so we will hunt for a near-prism with these special properties. Sometimes it is 
convenient to 
pin down the target even further: we will hunt for the ``lightest'' long near-prism, the lexicographically earliest of 
all shortest long near-prisms.

To do this we will first guess its ordered frame: so now we need a poly-time algorithm that, given an ordered frame, will test 
whether there is a lightest long near-prism with this ordered frame. This comes in two phases:
\begin{itemize}
\item 
Given an ordered frame, we generate a ``cleaning list'' of polynomially many sets of vertices, such that for every
shortest long near-prism $K$ of $G$ with the given ordered frame, there exists $X$ in the list such that $K$ is clean in $G\setminus X$; that is,
$X$ is disjoint from $V(K)$, and $X$ contains all $K$-major vertices. This is explained in section \ref{sec:kmajor}.
\item For each $X$ in this cleaning list, we search for a clean lightest long near-prism with the given ordered frame in $G\setminus X$.
This algorithm is explained in the remainder of this section.
\end{itemize}

A long near-prism $K$ in $G$ is {\em tidy} if $F_K^*$ is anticomplete to $V(G) \setminus V(K)$.
If we have a frame $F$ in $G$, and we are trying to test if there is a long near-prism $K$ with $F_K=F$, we might as well delete all vertices of $G$
not in $V(F)$ that have a neighbour in $F^*$, because no such vertex belongs to $V(K)$. If $G'$ is the graph that remains, and the 
long near-prism we are looking for exists, then it is tidy in $G'$.

\begin{lem}\label{prismjump}
Let $G$ be a prospect, and 
let $K$ be a tidy shortest long near-prism in $G$, with
constituent paths $P_1,P_2,P_3$.
For all distinct $i,j\in \{1,2,3\}$, there is no induced path $Q$ of $G$ with one end in $V(P_i)$ and the other 
in $V(P_j)$, such that 
\begin{itemize}
\item $V(Q)$ is 
anticomplete to $V(P_k)$ where $k\in \{1,2,3\}\setminus \{i,j\}$;
\item no vertex of $Q^*$ is $K$-major; and
\item $2|E(Q)|\le 1+\min(|E(P_i)|,|E(P_j)|)$.
\end{itemize}
\end{lem}

\begin{proof}
Suppose that there is such a path $Q$ for some shortest long near-prism $K$ in $G$, and choose $Q,K$ with minimal union. Let
the vertices of $Q$ be $q_0\dd q_1\cc q_t\dd q_{t+1}$ where $q_0\in V(P_i)$ and $q_{t+1}\in V(P_j)$. From the minimality of $Q\cup K$,
none of $q_2\ll q_{t-1}$ has a neighbour in $V(K)$. Since $q_1\ll q_t$ are not $K$-major, it follows that $t\ge 2$, and there is a 
three-vertex path of $K$ that contains all neighbours of $q_t$ in $V(K)$; and since $K$ is tidy, all neighbours of 
$q_t$ in $V(K)$ belong to $V(P_j)$, and so there is a minimal subpath $R_j$ of $P_j$ containing all neighbours of $q_t$ in $V(K)$.
Thus $R_j$ has length zero, one or two, and we will treat these cases separately. Let the ends of $R_j$ be $u_j,v_j$, 
where $a_j, u_j, v_j, b_j$ are in order in $P_j$ (in the usual notation). Since $K$ is tidy, the paths $a_j\dd P_j\dd u_j$ and $v_j\dd P_j\dd b_j$ both have length at least $\ell/2-1$. Similarly, there is a minimal subpath
$R_i$ of $P_i$ containing all neighbours of $q_i$ in $V(K)$, of length at most two, with ends $u_i,v_i$ say. 
Let $k\in \{1,2,3\}\setminus \{i,j\}$.
\\
\\
(1) {\em $R_j$ does not have length one.}
\\
\\
Suppose it does. Let $S$ be the induced path from $q_1$ to $a_i$ with interior in $V(P_i)$; then there is a prism with bases
$\{a_1,a_2,a_3\}$, $\{q_t, u_j, v_j\}$ and constituent paths 
$$a_i\dd S\dd q_1\cc q_{t},$$
$$a_j\dd P_j\dd u_j,$$
$$a_k\dd P_k\dd b_k\dd b_j\dd P_j\dd v_j.$$
All of its holes are long, since
$2(\ell/2-1)+ 4\ge \ell$, and so it is a long prism, and hence not shorter than $K$. Consequently
$|E(S)|+t-1\ge |E(P_i)|$. Similarly, let $T$ be the induced path from $q_i$ to $b_i$ with interior in $V(P_i)$;
then $|E(T)|+t-1\ge |E(P_i)|$. Adding, we obtain that $|E(S)|+|E(T)|+(2t-2)\ge 2|E(P_i)|$. But $|E(S)|+|E(T)|\le |E(P_i)|+2$, and
so $2t\ge |E(P_i)|$, contrary to the hypothesis, since $|E(Q)|=t+1$. This proves (1).
\\
\\
(2) {\em $R_j$ does not have length zero.}
\\
\\
Suppose it does; so $u_j=v_j=q_{t+1}$. By (1) with $P_i,P_j$ exchanged, it follows that either $u_i=v_i$ or $u_i,v_i$ are nonadjacent.
If $u_i=v_i$ there is a long theta with constituent paths 
$Q$ and
$$q_0\dd P_i\dd a_i\dd a_j\dd P_j\dd q_{t+1}$$ 
$$q_0\dd P_i\dd b_i\dd b_j\dd P_j\dd q_{t+1},$$ 
a contradiction. If $u_i,v_i$ are nonadjacent, there is a long theta with constituent
paths 
$$q_1\cc q_t,$$ 
$$q_1\dd u_i\dd P_i\dd a_i\dd a_j\dd P_j\dd q_{t+1}$$ 
$$q_1\dd v_i\dd P_i\dd b_i\dd b_j\dd P_j\dd q_{t+1},$$
a contradiction. This proves (2).

\bigskip

From (1) and (2) we may assume that $R_i,R_j$ both have length two. 
Let $P_j'$ be the path obtained from $P_j$ by replacing the subpath $R_j$
by $u_j\dd q_t\dd v_j$. Then $P_i,P_j', P_k$ are the constituent paths of a shortest long near-prism $K'$, also with frame $F_K$.
From the minimality of $Q\cup K$,
one of $q_1\ll q_{t-1}$ is $K'$-major. Since it is not $K$-major, it is adjacent to $q_t$, and hence must be $q_{t-1}$.
Thus $q_{t-1}$ has a neighbour in $V(K)$, and so $t=2$ (because none of $q_2\ll q_{t-1}$ has a neighbour in $V(K)$).
But then the subgraph induced on $V(P_i\cup P_j')\cup \{q_1\}$ is a long ban-the-bomb, a contradiction.
This proves Lemma \ref{prismjump}.

\end{proof}

If $F$ is a frame with bases $A,B$, and $v$ is an end of $F$, choose $u\in V(A\cup B)$ with minimum $F$-distance to $v$; we call
$v$ the {\em $u$-end} of $K$.
For each $u\in V(A\cup B)$ there is at most one $u$-end of $F$, but there might be none.
For a path $P$ with ends $a,b$, we call $v \in V(P)$ a {\em midpoint} of $P$ if
$|d_P(v,a)- d_P(v,b)|\le 1$.

We would like to assign weights to the edges of $G$, all very close to one and all different, such that no two different sets of 
edges $X,Y$ have the same total weight, and if $|X|<|Y|$ then the total weight in $X$ is less than that in $Y$. 
A convenient way to do this, and a way that is easy to handle algorithmically, is to 
take an arbitrary linear ordering of $E(G)$, say $E(G)=\{e_1\ll e_n\}$, and let edge $e_i$ have weight $1+2^{-i}$ for each $i$; 
then the total
weight in a set $X$ is less than that in a set $Y$ if and only if either $|X|<|Y|$, or $|X|=|Y|$ and $X$ is lexicographically 
earlier than $Y$ (the latter means that $Y$ contains $e_i$
where $i\in \{1\ll n\}$ is minimum with $e_i\in (X\setminus Y)\cup (Y\setminus X)$). So, let us take some linear order of $E(G)$,
and for $X,Y\subseteq E(G)$, we say $X$ is {\em lighter} 
than $Y$ if either $|X|<|Y|$, or $|X|=|Y|$ and $X$ is lexicographically earlier than $Y$.
If $G$ has a long near-prism, it has at least one shortest long near-prism, and exactly one of them is the lightest long near-prism; and we
find that for algorithms it is better to hunt for the lightest long near-prism than just a shortest one. 
These weights have $O(|G|^2)$ bits, so doing arithmetic
with them is a little time-consuming; but we can certainly find the lightest $st$-path in time $\mathcal{O}(|G|^3)$ (and if it mattered,
we could do it faster).

Let us say an $st$-path $P$ in a graph $G$ is {\em locally lightest} if for every $st$-path $Q$
that is lighter than $P$, some vertex $v$ of $Q$ satisfies $\max(d_{G}(s,v), d_G(t,v))> |E(P)|/2$. It follows immediately
that there is no such path $Q$, because 
$$d_G(s,v)+d_G(t,v)\le |E(Q)|\le |E(P)|$$
for every vertex $v$ of $V(Q)$, and therefore every locally lightest $st$-path is the (unique) lightest $st$-path.

\begin{theorem} \label{thm:detectcleanprism}
There is an algorithm with the following specifications:
\begin{description}
\item[Input:] A prospect $G$, a linear order of $E(G)$, and an ordered frame $\mathcal{F}$ in $G$.
\item[Output:] Decides either that $G$ contains a long near-prism with ordered frame $\mathcal{F}$, or that there is a no long near-prism 
that is the lightest among all long near-prisms, and has ordered frame $\mathcal{F}$, and is clean.
\item[Running time:] $\mathcal{O}(|G|^{3})$.
\end{description}
\end{theorem}

\begin{proof}
Here is the algorithm. Let the frame $F$ of $\mathcal{F}$ have bases $\{a_1,a_2,a_3\}$ and $\{b_1,b_2,b_3\}$, where the linear orders
of $\mathcal{F}$ are $a_1<a_2<a_3$ and $b_1<b_2<b_3$. For $1\le i\le 3$, let $s_i$ be the $a_i$-end of $F$ and let $t_i$ be the 
$b_i$-end of $F$, if they exist. (Certainly $s_2,t_2,s_3,t_3$ exist, but $s_1,t_1$ might not.)
Let $W_0$ be the set of all vertices of $G$ that are not ends of $F$, and belong to or have a neighbour in $F^*$, and let 
$G_0=G\setminus W_0$.
\begin{enumerate}[\bf Step 1:]
\item If $s_1$ is defined, compute the lightest $s_1t_1$-path $M_1$ in $G_0$ 
(if there is no such path, output ``failure'', that is, the desired near-prism does not exist, and stop).
If $s_1$ is not defined, let $M_1$
be the null graph.  In either case let $W_1$ be the set of vertices of $G_0$ that belong to or
have a neighbour in $V(M_1)$, and let $G_1=G_0\setminus W_1$.
\item Compute the lightest $s_2t_2$-path $M_2$ in $G_1$ 
(reporting failure if there is no such path).
Let $W_2$ be the set of vertices of $G_1$ that belong to or
have a neighbour in $V(M_2)$, and let $G_2=G_1\setminus W_2$.
\item Compute the lightest $s_3t_3$-path $M_3$ in $G_2$ (reporting failure if there is no such path).
\item Check whether $F\cup M_1\cup M_2\cup M_3$ is a long near-prism in $G$, and if so,
output that  fact and stop.
\end{enumerate}
This concludes the description. For running time, 
we just have to find the sets $W_0, W_1,W_2$, which take time $\mathcal{O}(|G|^2)$, and solve three lightest-path problems,
so the total running time is $\mathcal{O}(|G|^3)$.

To prove correctness: the positive output is clearly correct, but we need to check the negative output. Assume then
that there is a a long near-prism $K$ that is the lightest among all near-prisms, and it has
ordered frame $\mathcal{F}$, and is clean. Let its constituent paths be $P_1,P_2,P_3$, numbered according to $\mathcal{F}$,
and hence with $|E(P_1)|\le |E(P_2)|\le |E(P_3)|$.

We claim that in step 1 above, the algorithm will compute some $M_1$, and if $s_1,t_1$ exist then $M_1$ is the path 
$s_1\dd P_1\dd t_1$. To see this, the claim is true if $s_1,t_1$ do not exist, so we assume they do.
Then there is an $s_1t_1$-path in $G_0$, namely the path $s_1\dd P_1\dd t_1$, and so the algorithm
will not report failure in step 1, and so computes the lightest $s_1t_1$-path $M_1$
in $G_0$. But $s_1\dd P_1\dd t_1$ is a locally lightest $s_1t_1$-path
in $G_0$, because of
Lemma \ref{prismjump}, and therefore equals $M_1$.
This proves our claim.

Similarly in step 2, the algorithm computes $M_2$, and if $s_2,t_2$ exist then $M_2$ is the path 
$s_2\dd P_2\dd t_2$, because $s_2\dd P_2\dd t_2$ is locally lightest in $G_1$ (though not necessarily in $G_0$).
And in step 3 the algorithm computes $M_3$; and $F\cup M_1\cup M_2\cup M_3$ is a long near-prism, and the output is correct.
This proves Theorem \ref{thm:detectcleanprism}.

\end{proof}

When the algorithm of Theorem \ref{thm:detectcleanprism} finds a long near-prism $K$, it is tempting to claim that $K$ is the lightest
long near-prism with ordered frame $\mathcal{F}$. But that might not be true; perhaps some vertices of $K$ are $K'$-major, where $K'$
is the lightest long near-prism with ordered frame $\mathcal{F}$, and then the algorithm might find $K$ instead of $K'$.

\section{Major vertices on near-prisms}

In this section we prove some properties of $K$-major vertices, when $K$ is a shortest long near-prism.
If $K$ is a long near-prism with constituent paths $P_1,P_2,P_3$, and each $P_i$ has ends $a_i, b_i$
as usual, and $x$ is $K$-major with a neighbour in $V(P_i)$, we define 
$\alpha_i(x)$ to be the neighbour $v$ of $x$ in $V(P_i)$ such that the path $v\dd P_i\dd a_i$ is minimal; and define $\beta_i(x)$
to be the neighbour $v$ of $x$ in $V(P_i)$ such that the path $v\dd P_i\dd b_i$ is minimal.
We begin with some lemmas:

\begin{lem} \label{2paths}
Let $K$ be a tidy shortest long near-prism in a graph $G$. If $x$ is a $K$-major vertex, 
then $x$ has neighbours in at least two constituent paths of $K$.
\end{lem}

\begin{proof}
In the usual notation, suppose that all neighbours of $x$ in $V(K)$ are contained in $V(P_1)$, say; so $\alpha_1(x)\dd P_1\dd \beta_1(x)$ has length 
strictly greater than two. We obtain a near-prism $K'$ shorter than $K$ by replacing $\alpha_1(x)\dd P_1\dd \beta_1(x)$ in $P_1$ with the 
path $\alpha_1(x)\dd x\dd \beta_1(x)$. Since $K'$ contains the same frame as $K$, it follows that $K'$ is a long near-prism, 
a contradiction. This proves Lemma \ref{2paths}.
\end{proof}

\begin{lem}\label{notwohat}
Let $K$ be a tidy shortest long near-prism in a graph $G$, with constituent paths $P_1,P_2,P_3$. For all distinct $i,j\in \{1,2,3\}$, if
$x$ is a $K$-major vertex
with no neighbours in $V(P_j)$, then $x$ either has exactly one neighbour in $V(P_i)$, or two nonadjacent neighbours in $V(P_i)$.
\end{lem}

\begin{proof}

Suppose that $x$ has no neighbour in $V(P_3)$, and $\alpha_1(x), \beta_1(x)$ are distinct and adjacent, say.
Then there is a long prism with bases $\{a_1,a_2,a_3\}$ and
$\{x,\alpha_1(x), \beta_1(x)\}$ and constituent paths 
$$a_1\dd P_1\dd \alpha_1(x),$$
$$a_2\dd P_2\dd \alpha_2(x),$$
$$a_3\dd P_3\dd b_3\dd b_1\dd P_1\dd \beta_1(x),$$ 
and it is shorter than $K$, a contradiction. This proves Lemma \ref{notwohat}.

\end{proof}

\begin{lem}\label{3pairwisenonadjacent}
Let $G$ be a graph with no long theta, and let $K$ be a tidy shortest long near-prism in $G$.
If $x$ is a $K$-major vertex, then $x$ has three pairwise non-adjacent neighbours in $V(K)$.
\end{lem}
\begin{proof}
Suppose not.
By Lemma \ref{2paths}, in the usual notation we may assume $x$ has a neighbour in $V(P_1)$ and a neighbour in $V(P_2)$, and we may assume that 
$x$ has no neighbours in $V(P_3)$. If $x$ has exactly one neighbour in $V(P_1)$ and exactly one neighbour in $V(P_2)$,
then $V(P_1 \cup P_2) \cup \{x\}$ induces a long theta. 
So by Lemma \ref{notwohat} we may assume that $x$ has two nonadjacent 
neighbours in $V(P_1)$; but then the claim is true.
This proves Lemma \ref{3pairwisenonadjacent}.
\end{proof}

\begin{lem}\label{notwotwos}
Let $G$ be a graph with no long theta, and let $K$ be a tidy shortest long near-prism in $G$, with constituent paths $P_1,P_2,P_3$. 
Let $x, y$ be nonadjacent $K$-major vertices. If $i,j\in \{1,2,3\}$, and $x$ has no neighbours in $V(P_i)$ and $y$ has no neighbours in $V(P_j)$
then $i = j$.
\end{lem}
\begin{proof}
Suppose that $x$ has no neighbours in $V(P_3)$, and $y$ has no neighbours in $V(P_1)$, say. Then $x$ has neighbours in $V(P_1)$ and 
in $V(P_2)$, and $y$ has neighbours in $V(P_2)$ and $V(P_3)$ by Lemma \ref{2paths}. Let $M$ be an induced $xy$-path with interior 
in $V(P_2)$. By Lemma \ref{notwohat}, $\alpha_1(x), \beta_1(x)$ are either equal or nonadjacent, and $\alpha_3(y), \beta_3(y)$
are either equal or nonadjacent.
Thus there are four cases, but in each case there is a long theta induced on the union of the vertex sets of the paths
$\alpha_1(x)\dd P_1\dd a_1$, $\beta_1(x)\dd P_1\dd b_1$, $\alpha_3(y)\dd P_3\dd a_3$, $\beta_3(y)\dd P_3\dd b_3$ and $M$, a contradiction.
This proves Lemma \ref{notwotwos}.

\end{proof}

We need some more definitions. 
Let $K$ be a  tidy shortest long near-prism in $G$.
For $v\in V(K)\setminus F_K^*$ and integers $m,n\ge 0$, we define the path $K^m_n(v)$ as follows. In the usual notation, 
let $v\in V(P_i)$ say.
Let $M$ be the maximal subpath of the path $v\dd P_i\dd a_i$ that has one end $v$ and has length at most $m$,
and has no internal vertex in $F_K^*$. (Thus, $M$ is permitted to have an end in $F_K^*$, but no more.) Let $N$ be
the maximal subpath of the path $v\dd P_i\dd b_i$ that has one end $v$ and has length at most $n$,
and has no internal vertex in $F_K^*$; and let $K^m_n(v)=M\cup N$.

Also, if $x$ is $K$-major, then for $1\le i\le 3$, 
if $x$ has a neighbour in $V(P_i)$ let
$A_i(x)$ be the vertex set of the path $\alpha_i(x)\dd P_i\dd a_i$, and if $x$ has no such neighbour let $A_i(x)=V(P_i)$.
For $i,j\in \{1,2,3\}$, let $A_{i,j}(x) = A_i(x)\cup A_j(x)$, and let $A_{1,2,3}(x) = A_1(x)\cup A_2(x)\cup A_3(x)$.
If $x,y$ are $K$-major, we say that $y$
is 
{\em distant} from $x$ if 
\begin{itemize}
\item $x,y$ are nonadjacent, and  $y$ has a neighbour in $A_{1,2,3}(x)$;
\item for $1\le i\le 3$, if $x$ has a neighbour in $V(P_i)$, then $y$ has no neighbour in 
$V(K^{\ell-2}_1(\alpha_i(x)))$;  and
\item for $1\le i\le 3$, if $x$ has no neighbour in $V(P_i)$, then for some $j\in \{1,2,3\}\setminus \{i\}$, $y$ has no neighbour in 
$V(K^0_{\ell-3}(\beta_j(x)))$.
\end{itemize}

We need to prove some properties of distant pairs. 

\begin{lem}\label{distant}
Let $G$ be a graph with no long theta, and let $K$ be a tidy shortest long near-prism in $G$, with constituent paths $P_1,P_2,P_3$.
Let $x,y$ be $K$-major, where $y$ is distant from $x$. Then $y$ has exactly two neighbours in $A_{1,2,3}(x)$ and they are adjacent.
\end{lem}

\begin{proof}
We begin with:
\\
\\
(1) {\em If $x$ has a neighbour in each of $V(P_1), V(P_2), V(P_3)$ then the theorem holds.}
\\
\\
Suppose that $x$ has a neighbour in each of $V(P_1), V(P_2), V(P_3)$. If $y$ has a neighbour in each of $A_1(x)$, $A_2(x)$, $A_3(x)$,
there is a long theta formed by three $xy$-paths  with interiors in $A_1(x), A_2(x), A_3(x)$, a contradiction. So we may assume
that $y$ has no neighbour in $A_3(x)$. Suppose that $y$ also has no neighbour in $A_2(x)$. By Lemma \ref{2paths}, 
$y$ has a neighbour in one of $V(P_2), V(P_3)$, say $V(P_2)$. Let $M$ be an induced $xy$-path with interior in $V(P_2)$.
If $y$ has a unique neighbour in $A_1(x)$, there is a long theta induced
on $A_1(x)\cup A_3(x)\cup V(M)$. If $y$ has two nonadjacent neighbours in $A_1(x)$, there is an induced $\alpha_1(x)a_1$-path $R$ with
interior in $A_1(x)\cup \{y\}$ containing $y$, and then there is a long theta induced on $V(R)\cup A_3(x)\cup V(M)$, a contradiction.
So $y$ has exactly two adjacent neighbours in $A_1(x)$ and the theorem holds. 

\begin{figure}[H]
\centering
\begin{tikzpicture}[scale=.8,auto=left]
\def\t{.5}
\tikzstyle{every node}=[inner sep=1.5pt, fill=black,circle,draw]
\node (a2) at (-5,0) {};
\node (a1) at (-6,2) {};
\node (a3) at (-6,-2) {};
\node (b2) at (5,0) {};
\node (b1) at (6,2) {};
\node (b3) at (6,-2) {};
\node (x) at (0,.7) {};
\node (c2) at (-2,0) {};
\node (c1) at (-2,2) {};
\node (c3) at (-2,-2) {};
\node (d2) at (2,0) {};
\node (d1) at (2,2) {};
\node (d3) at (2,-2) {};

\foreach \from/\to in {a1/a2,a1/a3,a2/a3,b1/b2,b1/b3,b2/b3,x/c1,x/c2,x/c3,x/d1,x/d2,x/d3}
\draw [-] (\from) -- (\to);
\foreach \from/\to in {a1/b1,a2/b2,a3/b3}
\draw [dashed] (\from) -- (\to);
\draw (x) -- ({\t*cos(60)},{.7+\t*sin(60)});
\draw (x) -- ({\t*cos(90)},{.7+\t*sin(90)});
\draw (x) -- ({\t*cos(120)},{.7+\t*sin(120)});
\draw (x) -- ({\t*cos(240)},{.7+\t*sin(240)});
\draw (x) -- ({\t*cos(270)},{.7+\t*sin(270)});
\draw (x) -- ({\t*cos(300)},{.7+\t*sin(300)});

\tikzstyle{every node}=[]
\draw (x) node [left = 1mm]           {$x$};
\draw (a1) node [left]           {$a_1$};
\draw (a2) node [left]           {$a_2$};
\draw (a3) node [left]           {$a_3$};
\draw (b1) node [right]           {$b_1$};
\draw (b2) node [right]           {$b_2$};
\draw (b3) node [right]           {$b_3$};
\draw (c1) node [above]           {$\alpha_1(x)$};
\draw (c2) node [above]           {$\alpha_2(x)$};
\draw (c3) node [below]           {$\alpha_3(x)$};
\draw (d1) node [above]           {$\beta_1(x)$};
\draw (d2) node [above]           {$\beta_2(x)$};
\draw (d3) node [below]           {$\beta_3(x)$};
\node (A1) at (-4,2.3) {$A_1(x)$};
\node (A2) at (-3.5,.3) {$A_2(x)$};
\node (A3) at (-4,-1.7) {$A_3(x)$};

\end{tikzpicture}

\caption{$x$ has a neighbour in each of $P_1,P_2,P_3$ (possibly $\alpha_i(x)=\beta_i(x)$).} \label{fig:step1}
\end{figure}
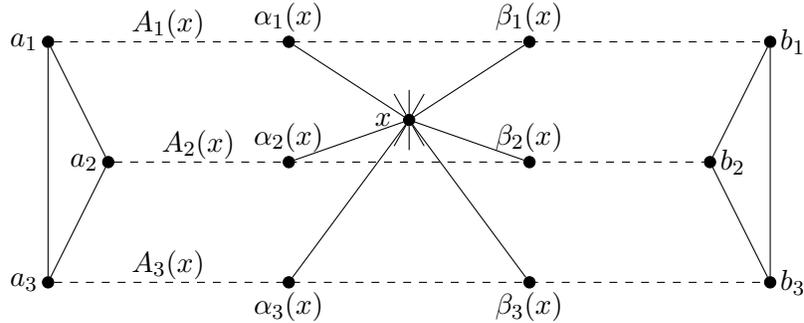

So we may assume that $y$ has a neighbour in $A_1(x)$ and in $A_2(x)$, 
and not in $A_3(x)$. If $y$ has two nonadjacent neighbours in $A_1(x)$, or two nonadjacent neighbours in $A_2(x)$, there
is a long theta formed by three $xy$-paths all with interior in $A_{1,2,3}(x)$, a contradiction. So by Lemma 
\ref{3pairwisenonadjacent},there exists $i\in \{1,2,3\}$ such that $y$ has a neighbour in $V(P_i)\setminus A_i(x)$. This neighbour is
nonadjacent to $\alpha_1(x),\alpha_2(x)$, obviously if $i = 3$ and from the definition of ``distant'' if $i \in \{1,2\}$.
Hence there is an induced $xy$-path $R$
with interior in $(V(P_i)\setminus A_i(x))\cup V(P_3)$, containing no neighbour of $\alpha_1(x)$ or $\alpha_2(x)$.
But then there is a long theta formed by the path $R$ and two $xy$-paths with interiors in $A_1(x), A_2(x)$ respectively, 
a contradiction. This proves (1).

\bigskip

We may therefore assume that $x$ has no neighbour 
in $V(P_1)$, and hence $A_1(x)=V(P_1)$. By Lemma \ref{3pairwisenonadjacent}, $x$ has two nonadjacent neighbours in one of $V(P_2), 
V(P_3)$, say in $V(P_j)$ where $j\in \{2,3\}$; thus $\alpha_j(x), \beta_j(x)$ are distinct and nonadjacent.
By Lemma \ref{notwotwos}, $y$ has a neighbour in $V(P_i)$ for $i=2,3$.
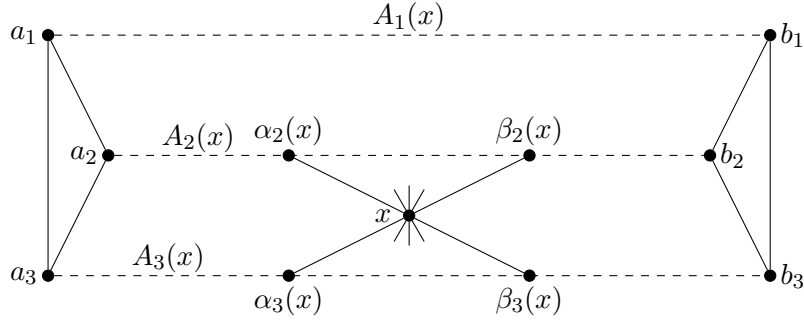
\begin{figure}[H]
\centering
\begin{tikzpicture}[scale=.8,auto=left]
\def\t{.5}
\tikzstyle{every node}=[inner sep=1.5pt, fill=black,circle,draw]
\node (a2) at (-5,0) {};
\node (a1) at (-6,2) {};
\node (a3) at (-6,-2) {};
\node (b2) at (5,0) {};
\node (b1) at (6,2) {};
\node (b3) at (6,-2) {};
\node (x) at (0,-1) {};
\node (c2) at (-2,0) {};
\node (c3) at (-2,-2) {};
\node (d2) at (2,0) {};
\node (d3) at (2,-2) {};

\foreach \from/\to in {a1/a2,a1/a3,a2/a3,b1/b2,b1/b3,b2/b3,x/c2,x/c3,x/d2,x/d3}
\draw [-] (\from) -- (\to);
\foreach \from/\to in {a1/b1,a2/b2,a3/b3}
\draw [dashed] (\from) -- (\to);
\draw (x) -- ({\t*cos(60)},{-1+\t*sin(60)});
\draw (x) -- ({\t*cos(90)},{-1+\t*sin(90)});
\draw (x) -- ({\t*cos(120)},{-1+\t*sin(120)});
\draw (x) -- ({\t*cos(240)},{-1+\t*sin(240)});
\draw (x) -- ({\t*cos(270)},{-1+\t*sin(270)});
\draw (x) -- ({\t*cos(300)},{-1+\t*sin(300)});

\tikzstyle{every node}=[]
\draw (x) node [left = 1mm]           {$x$};
\draw (a1) node [left]           {$a_1$};
\draw (a2) node [left]           {$a_2$};
\draw (a3) node [left]           {$a_3$};
\draw (b1) node [right]           {$b_1$};
\draw (b2) node [right]           {$b_2$};
\draw (b3) node [right]           {$b_3$};
\draw (c2) node [above]           {$\alpha_2(x)$};
\draw (c3) node [below]           {$\alpha_3(x)$};
\draw (d2) node [above]           {$\beta_2(x)$};
\draw (d3) node [below]           {$\beta_3(x)$};
\node (A1) at (-0,2.3) {$A_1(x)$};
\node (A2) at (-3.5,.3) {$A_2(x)$};
\node (A3) at (-4,-1.7) {$A_3(x)$};

\end{tikzpicture}

\caption{$x$ has no neighbour in $V(P_1)$.} \label{fig:step2}
\end{figure}
\noindent
(2) {\em $y$ does not have two nonadjacent neighbours in $A_2(x)\cup A_3(x)$.}
\\
\\
Suppose that it does; then there are two long $xy$-paths $R_1,R_2$, 
with $R_1^*, R_2^*\subseteq A_2(x)\cup A_3(x)$ and with $R_1^*$ anticomplete to $R_2^*$. 
If $y$
has a neighbour in $V(P_i)\setminus A_i(x)$ for some $i\in \{2,3\}$, 
this neighbour is nonadjacent to $\alpha_i(x)$ from the definition
of ``distant''; so if $y$
has a neighbour in $V(P_i)\setminus A_i(x)$ for some $i\in \{2,3\}$, or a neighbour in $V(P_1)$, there is an $xy$-path with 
interior in $V(K)\setminus (A_2(x)\cup A_3(x))$ with interior anticomplete to $R_1^*, R_2^*$, and these three paths form a long theta, 
a contradiction. So every neighbour of $y$ in $V(K)$ belongs to $A_2(x)\cup A_2(x)$. By Lemma \ref{2paths}, $y$
has a neighbour in $A_2(x)$ and one in $A_3(x)$, and by Lemma \ref{3pairwisenonadjacent} we may assume it has two nonadjacent neighbours
in $V(P_2)$; 
but then there is a long theta formed by the paths 
$$y\dd \beta_2(y)\dd P_2\dd \alpha_2(x)\dd x,$$ 
$$y\dd \beta_3(y)\dd P_3\dd \alpha_3(x)\dd x$$ 
$$y\dd \alpha_2(y)\dd P_2\dd a_2\dd a_1\dd P_1\dd b_1\dd b_j\dd P_j\dd \beta_j(x)\dd x,$$
a contradiction. This proves (2).
\\
\\
(3) {\em We may assume that $y$ has a neighbour in $V(P_1)$.}
\\
\\
Suppose that $y$ has no neighbour in $V(P_1)$.
We may therefore assume that $y$ has a unique neighbour $v\in A_{1,2,3}(x)$, because otherwise the theorem holds, and we may 
assume that $v\in A_2(x)$. Both $x,y$ have a neighbour in $V(P_2\cup P_3)$ that is not in $A_{2,3}(x)$ and has no neighbour in this set;
and so there is an $xy$-path $R$ with interior anticomplete to $A_{2,3}(x)$. But then there
is a long theta formed by the paths 
$$v\dd P_2\dd \alpha_2(x)\dd x,$$
$$v\dd y\dd R\dd x$$
$$v\dd P_2\dd a_2\dd a_3\dd P_3\dd \alpha_3(x)\dd x$$
a contradiction. This proves (3).
\\
\\
(4) {\em $y$ has no neighbour in $A_2(x)\cup A_3(x)$.}
\\
\\
Now $y$ has at most two neighbours in $A_2(x)\cup A_3(x)$.
If $y$ has two neighbours $u,v$ in $A_2(x)\cup A_3(x)$, then they are adjacent
and we may assume they belong to $A_2(x)$, and $a_2,u,v,\alpha_2(x)$ are in order in $P_2$. But then there is a long prism with bases
$\{a_1,a_2,a_2\}$ and $\{y,u,v\}$ and constituent paths 
$$y\dd \alpha_1(y)\dd P_1\dd a_1,$$ 
$$u\dd P_2\dd a_2$$  
$$v\dd P_2\dd \alpha_2(x)\dd x\dd \alpha_3(x)\dd P_3\dd a_3,$$ 
and it is shorter than $K$, a contradiction. Thus $y$ has at most one 
neighbour in $A_2(x)\cup A_3(x)$. If there is such a neighbour, say $v\in A_2(x)$, let $M$ be an induced $xy$-path with 
interior in 
$V(\beta_1(y)\dd P_1\dd b_1\dd b_j\dd P_j\dd \beta_j(x))$;
then there is a long theta formed by the 
paths 
$$v\dd P_2\dd \alpha_2(x)\dd x,$$
$$v\dd P_2\dd a_2\dd a_3\dd P_3\dd \alpha_3(x)\dd x,$$ 
$$v\dd y\dd M\dd x,$$
a contradiction. This proves (4).

\bigskip

From the definition of ``distant'', we may assume that $y$ has no neighbour in $V(K^0_{\ell-3}(\beta_2(x)))$.
Since $y$ has a neighbour in $V(P_1)$, there is an $xy$-path $R_1$ with one end $\beta_2(x)$ and with interior in
the vertex set of $\beta_1(y)\dd P_1\dd b_1\dd b_2\dd P_2\dd \beta_2(x)$, which is therefore long.
By Lemma \ref{notwotwos}, $y$ has a neighbour in $V(P_3)$ not in $A_3(x)$ and not adjacent to $\alpha_3(x)$; let $R_3$ be an 
induced $xy$-path with interior in $V(P_3)$, chosen with interior anticomplete to $\alpha_3(x)$ if $j=3$. Let $R_2$ be the path
$y\dd \alpha_1(y)\dd P_1\dd a_1\dd a_j\dd P_j\dd \alpha_j(x)$. If $\alpha_1(y)$ is distinct from and nonadjacent to $\beta_1(y)$, 
the three paths $R_1,R_2,R_3$ form a long theta, a contradiction. If $\alpha_1(y)=\beta_1(y)$, then the three paths 
$R_1\setminus \{y\}$,
$R_2\setminus \{y\}$, $\alpha_1(y)\dd y\dd R_3\dd x$ form a long theta, a contradiction. Thus $\alpha_1(y), \beta_1(y)$ are distinct and adjacent. This proves Lemma \ref{distant}.

\end{proof}
\begin{lem}\label{twotriangles}
Let $G$ be a graph with no long theta, and let $K$ be a tidy shortest long near-prism in $G$, with constituent paths $P_1,P_2,P_3$.
Let $x,y,z$ be $K$-major, such that $y,z$ are both distant from $x$ and $y,z$ are nonadjacent.
For all distinct $i,j,k\in \{1,2,3\}$, 
either there is no $yz$-path of length at least $\ell-2$ with interior in $A_{i,j}(x)$, 
or there is no $yz$-path of length at least $\ell-2$ with 
interior in $V(P_k)$.
\end{lem}

\begin{proof}
Suppose that $M_1$ is a $yz$-path of length at least $\ell-2$ with interior in $A_{i,j}(x)$, and $M_2$ is a $yz$-path of length at least $\ell-2$ with     
interior in $V(P_k)$.
By Lemma \ref{distant}, $y,z$ each have exactly two neighbours in $A_{i,j}(x)$ 
and they are adjacent. By Lemma \ref{notwohat}, $y,z$ each have a third neighbour in 
$V(P_i\cup P_j)$, and this neighbour does not belong to $A_{i,j}(x)$ and has no neighbour in $A_{i,j}(x)$, since $(x,y)$
and $(x,z)$ 
are distant. Consequently
there is an induced $yz$-path $M_3$ with interior in $V(P_i\cup P_j)$ and anticomplete to $M_1^*\cup M_2^*$; and $M_1,M_2,M_3$
form a long theta, a contradiction.  This proves Lemma \ref{twotriangles}.

\end{proof}

\begin{lem}\label{twoclosetwos}
Let $G$ be a graph with no long theta, and let $K$ be a tidy shortest long near-prism in $G$, with constituent paths $P_1,P_2,P_3$.
Let $x,y,z$ be $K$-major, such that $y,z$ are both distant from $x$. If there exist $i,j\in \{1,2,3\}$ such that $y$
has no neighbour in $V(P_i)$ and $z$ has no neighbour in $V(P_j)$ then $i = j$.
\end{lem}

\begin{proof}
Suppose that $y$ has no neighbour in $V(P_i)$ and $z$ has no neighbour in $V(P_j)$, and $i\ne j$. By Lemma \ref{notwotwos}, 
$y,z$ are adjacent; and also by Lemma \ref{notwotwos} (applied to $x,y$ and to $x,z$), $x$ has a neighbour
in each of $V(P_1), V(P_2), V(P_3)$. From Lemma \ref{distant}, $y,z$ each have exactly two neighbours
in $A_{1,2,3}(x)$ and they are adjacent, and from the symmetry we may assume that $y,z$ have no neighbour in $A_3(x)$.
Let $R$ be the path $\alpha_1(x)\dd P_1\dd a_1\dd a_2\dd P_2\dd \alpha_2(x)$, and let the neighbours of $y$ in $V(R)$ be $y_1,y_2$, where $\alpha_1(x), y_1,y_2,\alpha_2(x)$
are in order in $R$. Define $z_1,z_2$ similarly. We may assume that $\alpha_1(x),y_1,z_2,\alpha_2(x)$ are distinct and
in order in $R$. 
If the path $y_2\dd R\dd z_1$ has length at least $\ell-3$, then there is a long prism with bases $\{y,y_1,y_2\}$, 
$\{z,z_1,z_2\}$, and constituent paths 
$$y\dd z,$$
$$y_2\dd R\dd z_1,$$
$$y_1\dd R\dd \alpha_1(x)\dd x\dd \alpha_2(x)\dd R\dd z_2$$
and it is shorter than $K$, a contradiction. Thus $y_2\dd R\dd z_1$ has length at most $\ell-4$ and so 
$\{y_1,y_2,z_1,z_2\}$ is a subset of one of $A_1(x), A_2(x)$; and we may assume that $\{y_1,y_2,z_1,z_2\}\subseteq A_1(x)$.
So $y,z$ have no neighbours in $A_2(x)$ and no neighbours in $A_3(x)$, restoring the symmetry between $P_2,P_3$; and therefore
we may assume that $i=3$ and $j=2$, that is, $y$ has no neighbour in $V(P_3)$ and $z$ has no neighbour in $V(P_2)$. 
By Lemma \ref{2paths}, $y$ has a neighbour in $V(P_2)$ and $z$ has a neighbour in $V(P_3)$.

If $y_1=z_1$ and hence $y_2=z_2$, there is a 
long prism with bases $\{a_1,a_2,a_3\}$, $\{y,z,y_2\}$ and constituent paths 
$$y_2\dd P_1\dd a_1,$$ 
$$y\dd \alpha_2(y)\dd P_2\dd a_2$$ 
$$z\dd \alpha_3(z)\dd P_3\dd a_3,$$ 
and it is shorter than $K$, a contradiction. So $y_1\ne z_1$, and therefore 
$y_1,z_2$ are noadjacent.
Then there is a long theta with constituent paths 
$$z\dd y\dd y_1\dd P_1\dd \alpha_1(x)\dd x,$$
$$z\dd z_2\dd R\dd \alpha_2(x)\dd x,$$
and an induced $xz$-path with interior in $V(P_3)$, a contradiction. This proves Lemma \ref{twoclosetwos}.

\end{proof}

\section{Cleaning lightest long near-prisms}\label{sec:kmajor}

In this section we will complete the proof of Theorem \ref{alg:longprisms}, by showing how to compute a cleaning list for lightest
long near-prisms.

Let $\mathcal{Q}$ be a set of paths of $G$, pairwise anticomplete. We define $V(\mathcal{Q})$
to be the union of the vertex sets of the members of $\mathcal{Q}$, and $\mathcal{Q}^*$ to be the union of the interiors of the
member of $\mathcal{Q}$, and the {\em cost} of $\mathcal{Q}$ to be the cardinality of $V(\mathcal{Q})$.

Let $K$ be a shortest long near-prism, with an ordered frame $\mathcal{F}$, and with constituent paths $P_1,P_2,P_3$, numbered 
according to $\mathcal{F}$.
A $K$-major vertex $x$ is {\em $(K,\mathcal{F})$-extremal} if either
\begin{itemize}
\item there is a $K$-major vertex with no neighbour in $V(P_1)$, and $x$ is chosen with no neighbour in $V(P_1)$ and with $A_2(x)$ maximal; or
\item every $K$-major vertex has a neighbour in $V(P_1)$, and $x$ is chosen with $A_1(x)$ maximal.
\end{itemize}
Thus if $x$ is $(K,\mathcal{F})$-extremal, and has a neighbour in $V(P_1)$, then
every $K$-major vertex has a neighbour in $A_1(x)$; and otherwise $A_1(x)=V(P_1)$, and every $K$-major vertex has a neighbour in
$V(P_1)\cup A_2(x)$.
A {\em $(K,\mathcal{F})$-contrivance} 
consists of a quintuple $(x,y,\alpha, h,\mathcal{Q})$, where $x,y$ are $K$-major (possibly $y=x$), and $x$ is 
$(K,\mathcal{F})$-extremal, and 
$\mathcal{Q}$ is a set of paths 
of $K$, pairwise anticomplete, and $\alpha\in \mathcal{Q}^*$, and $h\in \{1,2\}$, such that:
\begin{itemize}
\item every $K$-major vertex is either adjacent to one of $x,y$ or has a neighbour in $\mathcal{Q}^*$; 
\item if $x$ has a neighbour in $V(P_1)$ then $h=1$ and $\alpha=\alpha_1(x)$, and otherwise $h=2$ and $\alpha=\alpha_2(x)$; and
\item every neighbour of $x$ or $y$ in $A_{1,2}(x)$ belongs to $\mathcal{Q}^*$.
\end{itemize}
Its {\em cost} is the cost of $\mathcal{Q}$.
From Lemma \ref{twotriangles} we have:
\begin{lem} \label{contrivance}
Let $G$ be a prospect, and let $K$ be a tidy shortest long near-prism in $G$ with an ordered frame $\mathcal{F}$, and 
with a $K$-major vertex.
Then there is a $(K,\mathcal{F})$-contrivance with cost at most $6\ell-2$.
\end{lem}

\begin{proof}
Let $P_1,P_2,P_3$ be the constituent paths of $K$. 
Choose $x$  $(K,\mathcal{F})$-extremal, and let $S$ be the set of all $K$-major vertices that are distant from $x$. 

If $x$ has a neighbour in $V(P_1)$ let $h=1$ and $\alpha=\alpha_1(x)$, and otherwise let $h=2$ and $\alpha=\alpha_2(x)$.
If $x$ has a neighbour in $V(P_i)$ for $i = 1,2,3$, let 
$Q_i$ be the path $K^{\ell-1}_2(\alpha_i(x))$ for $i = 1,2,3$. If $x$ has neighbours in $V(P_i), V(P_j)$ and not in $V(P_k)$, 
where $\{i,j,k\}=\{1,2,3\}$ and $i<j$, let $Q_1$ be the path $K^{\ell-1}_2(\alpha_i(x))$, let 
$Q_2$ be the path $K^{\ell-1}_2(\alpha_j(x))$,
and let $Q_3$ be the path $K^{1}_{\ell-2}(\beta_i(x))$.
Every $K$-major vertex 
has a neighbour in $A_{1,2,3}(x)$, since $x$ is  $(K,\mathcal{F})$-extremal; and so every $K$-major vertex nonadjacent to $x$ 
either belongs to $S$ or 
has a neighbour in 
one of $Q_1^*, Q_2^*,Q_3^*$, from the definition of ``distant''. If $S=\emptyset$, let $\mathcal{Q}$ be the set of
components of the graph induced on the union of the vertex sets of $Q_1,Q_2,Q_3$; then $(x,x,\alpha,h,\mathcal{Q})$
is a $(K,\mathcal{F})$-contrivance satisfying the theorem, so we may assume that $S\ne \emptyset$.

If every vertex in $S$ has a neighbour in $V(P_3)$, let $k=3$, and otherwise let $k=2$; then by Lemma \ref{twoclosetwos},
every vertex in $S$ has a neighbour in $V(P_k)$. Choose $y\in S$
with $A_k(y)$ maximal, let $Q_4$ be the path
$K^{\ell-4}_1(\alpha_k(y))$ and let $Q_5$ be a path of $K$ of length $2\ell-7$ such that the two neighbours of $y$
in $A_{1,2}(x)$ are the two middle vertices of $Q_5$.
\\
\\
(1) {\em Every vertex in $S$ nonadjacent to $y$ has a neighbour in $Q_4^*\cup Q_5^*$.}
\\
\\
Let $z\in S$ be nonadjacent to $y$,
and suppose it has no neighbour in $Q_4^*\cup Q_5^*$. From the choice of $x$, it follows that $y,z$ both have a neighbour 
in $A_{1,2}(x)$, and so there is a $yz$-path $M_1$ of length at least $\ell-2$ with 
interior in $A_{1,2}(x)$; and $M_1$ has length at least $\ell-2$ since $z$ has no neighbour in $Q_5^*$. 
By Lemma \ref{twotriangles}, there is no $yz$-path of length at least $\ell-2$
with interior in $A_3(x)$. Suppose that $k=3$; then from the choice of $y$, 
there is a $yz$-path with interior in $A_3(y)$,
which has length at least $\ell-2$ since $z$ has no neighbour in $Q_4^*$, a contradiction. So $k=2$, and therefore some vertex
in $S$ has no neighbour in $V(P_3)$; and so by Lemma \ref{notwotwos}, $x$ and $z$ both have a neighbour in $V(P_1)$ and in
$V(P_2)$; and $z$ has a neighbour in $A_2(y)$ from the choice of $y$.
Hence there is a $yz$-path $M_2$  with interior in $A_2(y)$,
which has length at least $\ell-2$ since $z$ has no neighbour in $Q_4^*$. But $x$
has a neighbour in $V(P_1)$, and therefore $y,z$ both have neighbours in $A_1(x)$ since $x$ is $(K,\mathcal{F})$-extremal, 
and so $y,z$ have no neighbours in $A_2(x)$ by
Lemma \ref{distant}; and it follows that $M_1$ has interior in $V(P_1)$. This contradicts Lemma \ref{twotriangles},
(taking $i = 1$, $j=3$ and $k=2$). This proves (1).

\bigskip

Let $\mathcal{Q}$ be the set of
components of the graph induced on the union of the vertex sets of $Q_1\ll Q_5$; then $(x,y,\alpha,h,\mathcal{Q})$
is a $(K,\mathcal{F})$-contrivance satisfying the theorem.
This proves Lemma \ref{contrivance}.

\end{proof}
If $K$ is a tidy long near-prism, and $\mathcal{F}$ is an ordered frame for $K$, and the constituent paths of $K$ are $P_1,P_2,P_3$
numbered according to $\mathcal{F}$, and $x$ is $K$-major, let $L(x)=A_1(x)$ if $x$ has a neighbour in $V(P_1)$, and $L(x)=V(P_1)\cup A_2(x)$ otherwise.
If $K$ is a tidy, lightest long near-prism, then a knowledge of the ordered frame $\mathcal{F}$ and of a $(K,\mathcal{F})$-contrivance 
$(x,\alpha,h, \mathcal{Q})$
allows us to reconstruct $L(x)$, as 
the next result shows: 

\begin{lem}\label{firstpath}
Let $G$ be a prospect, let $K$ be a tidy lightest long near-prism in $G$, let $\mathcal{F}$ be an ordered frame of $K$, with frame $F$, and
let $(x,y,\alpha,h, \mathcal{Q})$ be a $(K,\mathcal{F})$-contrivance. 
Let $P_1,P_2,P_3$ be the constituent paths of $K$, numbered according to $\mathcal{F}$, 
where $P_i$ has ends $a_i, b_i$ as usual. Let $s_i, t_i$ be the $a_i$-end and $b_i$-end of $F$ respectively, 
if they exist.
Let $Z_1$ be the set of all vertices of $G$ not in
$V(\mathcal{Q})$ but with a neighbour in $\mathcal{Q}^*$, and let $Z_2$ be the set of all vertices adjacent to $x$ or $y$ that are 
not in $V(F)$ or in $V(\mathcal{Q})$.
Let $G_1=G\setminus (Z_1\cup Z_2)$.
\begin{itemize}
\item If $h=1$ (and therefore $x$ has a neighbour in $V(P_1)$, and $\alpha=\alpha_1(x)$, and $s_1$ is defined), 
then $s_1\dd P_1\dd \alpha_1$ is the lightest
$s_1\alpha$-path in $G_1$.
\item Assume that $h=2$ (and so $x$ has no neighbour in $V(P_1)$, and $\alpha=\alpha_2(x)$, and $s_2$ is defined).
If $s_1$ is not defined, then $P_1$ is the $a_1b_1$-path in $F\setminus \{a_2,a_3,b_2,b_3\}$, and $s_2\dd P_2\dd \alpha_2(x)$
is the lightest $s_2\alpha$-path in $G_1$. 
If $s_1$ is defined, then $s_1\dd P_1\dd t_1$ is the lightest $s_1t_1$-path in $G_1$, and 
$s_2\dd P_2\dd \alpha_2(x)$
is the lightest $s_2\alpha$-path in $G_2$, where $G_2$ is obtained from $G_1$ by deleting all vertices that belong to or have a neighbour in  $V(s_1\dd P_1\dd t_1)$.
\end{itemize}
\end{lem}

\begin{proof}

To prove the first bullet, we assume that $h=1$, and so $x$ has a neighbour in $V(P_1)$, and therefore $s_1,t_1$ are defined, and $\alpha=\alpha_1(x)$.
The path $s_1\dd P_1\dd \alpha_1(x)$ is the locally lightest $s_1\alpha_1(x)$-path in $G_1$ by Lemma \ref{prismjump}, and 
hence is the lightest 
$s_1\alpha$-path in $G_1$. This proves the first bullet.

For the second bullet, we assume that $h=2$, and so $x$ has no neighbour in $V(P_1)$, and therefore $x$ has a neighbour in $V(P_2)$ by 
Lemma \ref{2paths}; and so $s_2,t_2$ are defined and $\alpha=\alpha_2(x)$. If $s_1$ is not defined, then $P_1$ is a path of $F$
as claimed, and $s_2\dd P_2\dd \alpha_2(x)$ is the locally lightest $s_2\alpha_2(x)$-path in $G_1$ by Lemma \ref{prismjump}, and hence
is the lightest $s_2\alpha$-path in $G_1$. So we assume that $s_1,t_1$ are defined.
Then $s_1\dd P_1\dd t_1$ is a locally lightest $s_1t_1$-path in $G_1$, by Lemma \ref{prismjump},
and hence is the lightest $s_1t_1$-path in $G_1$. Similarly $s_2\dd P_2\dd \alpha_2(x)$ is a locally lightest $s_2\alpha_2(x)$-path
in $G_2$ (though not necessarily in $G_1$) by Lemma \ref{prismjump}, and so is the lightest $s_2\alpha_2(x)$-path
in $G_2$. This proves the second bullet and so proves 
Lemma \ref{firstpath}.

\end{proof}

Thus, if there is a lightest long near-prism $K$, with a given ordered frame $\mathcal{F}$ and a given $(K,\mathcal{F})$-contrivance 
$(x,y,\alpha,h,\mathcal{Q})$, we can reconstruct $L(x)$ algorithmically, using the construction of Lemma \ref{firstpath}, 
in time $\mathcal{O}(|G|^3)$. More exactly, if $h=1$, then the first bullet of Lemma \ref{firstpath} gives a method
to compute $A_1(x)=L(x)$. If $h=2$, we first compute $P_1$ using the method of the second bullet of  Lemma \ref{firstpath}; then compute $G_2$; and then compute $A_2(x)$, again using the method of the second bullet of  Lemma \ref{firstpath}.
In summary:

\begin{lem}\label{reconstruct}
There is an algorithm with the following specifications:
\begin{description}
\item[Input:]
A prospect $G$, a linear order of the edges of $G$, 
an ordered frame $\mathcal{F}$ in $G$, and a quintuple $(x,y,\alpha,h,\mathcal{Q})$ where $x,y,\alpha\in V(G)$ and $\mathcal{Q}$
is a set of pairwise anticomplete induced paths of $G$.
\item[Output:] A subgraph $L$ of $G$, such that if there is a long near-prism in $G$, and the lightest long near-prism $K$
is tidy and has ordered frame $\mathcal{F}$ and $(x,y,\alpha,h,\mathcal{Q})$ is a $(K,\mathcal{F})$-contrivance, then $L=L(x)$.
\item[Running time:] $\mathcal{O}(|G|^3)$.
\end{description}
\end{lem}

The good thing about having reconstructed $L(x)$ is that every $K$-major vertex has a neighbour in $L(x)$, either in the interior
of the path $G[L(x)]$ or in $\mathcal{Q}^*$; and no vertices not in $V(K)\setminus L(x)$ have such a neighbour, 
so now we can clean the $K$-major vertices. More exactly, let $Z_4$
be the set of all vertices of $G$ that are not in $V(F)\cup V(\mathcal{Q})\cup V(L(x))$ and have a neighbour either in
$\mathcal{Q}^*$ or in the interior of a path of $L(x)$; then $Z_4\cup V(K)=\emptyset$ and every $K$-major vertex belongs to $Z_4$.
We obtain:

\begin{lem}\label{cleanprism}
There is an algorithm with the following specifications:
\begin{description}
\item[Input:]
A prospect $G$, a linear order of the edges of $G$,
an ordered frame $\mathcal{F}$ in $G$, and a quintuple $(x,y,\alpha,h,\mathcal{Q})$ where $x,y,\alpha\in V(G)$ and $h\in \{1,2\}$, and $\mathcal{Q}$
is a set of pairwise anticomplete induced paths of $G$.
\item[Output:] A subset $X\subseteq V(G)$, such that if there is a long near-prism in $G$, and the lightest long near-prism $K$
is tidy and has ordered frame $\mathcal{F}$ and $(x,y,\alpha,h,\mathcal{Q})$ is a $(K,\mathcal{F})$-contrivance, then $X$
contains all $K$-major vertices and is disjoint from $V(K)$.
\item[Running time:] $\mathcal{O}(|G|^3)$.
\end{description}
\end{lem}

We can now prove Theorem \ref{alg:longprisms}, which we restate:
\begin{theorem}\label{alg:longprisms2}
There is an algorithm with the following specifications:
\begin{description}
\item[Input:] A prospect $G$.
\item[Output:] Decides whether $G$ contains a long near-prism.
\item[Running time:] $\mathcal{O}(|G|^{9\ell+3})$.
\end{description}
\end{theorem}
\begin{proof}
Fix a linear order of the edges of $G$.
Enumerate all ordered frames $\mathcal{F}$ in $G$. For each one, let $\mathcal{F}$ have frame $F$, and compute $G_1$, the graph obtained 
from $G$ by deleting all vertices not in $F^*$ but with a neighbour in $F^*$, except the ends of $F$. Compute the
linear order of $E(G_1)$ induced from the given linear order of $E(G)$. 
Compute all  quintuples $(x,y,\alpha,h,\mathcal{Q})$ where $x,y\in V(G_1)$, and $\alpha\in \mathcal{Q}^*$, and $h\in \{1,2\}$, 
and $\mathcal{Q}$
is a set of pairwise anticomplete induced paths of $G_1$ with cost at most $6\ell-2$. Apply the algorithm of Lemma \ref{cleanprism} 
to $G_1$, the linear order of $E(G_1)$, $\mathcal{F}$ and 
$(x,y,\alpha,h,\mathcal{Q})$, to obtain a set $X\subseteq V(G_1)$. Apply the algorithm of Theorem
\ref{thm:detectcleanprism} to $G_1\setminus X$, the induced linear order of its edge set, and the given frame. 
If this tells us that $G_1$ has a long near-prism, output this and stop.
If after examining all choices of $(x,y,\alpha,h,\mathcal{Q})$ we have not found a long near-prism, move to the next ordered frame;
and if after examining all ordered frames we have not found a long near-prism, report that there is none.

There are only at most $3\ell$ vertices in a frame, and so only $\mathcal{O}(|G|^{3\ell})$ different ordered frames to examine.
For each one, there are only $\mathcal{O}(|G|^{6\ell})$ different quintuples $(x,y,\alpha,h,\mathcal{Q})$ to check, since 
$\mathcal{Q}$ has cost at most $6\ell-2$ and there are only  at most $6\ell-2$ choices for $\alpha$. For each choice of the 
quintuple, applying  the algorithm of Lemma \ref{cleanprism} 
takes time $\mathcal{O}(|G|^3)$, and then applying the algorithm of Theorem
\ref{thm:detectcleanprism} takes time  $\mathcal{O}(|G|^3)$. So the total running time is  $\mathcal{O}(|G|^{9\ell+3})$.

For correctness, certainly if the algorithm reports a long near-prism then this is correct. To check the converse, suppose that $G$
contains a long near-prism, and let $K$ be the lightest long near-prism. Let $\mathcal{F}$ be an ordered frame for $K$. Since $K$
has a tidy frame in $G_1$, 
Lemma \ref{contrivance} implies that there is a $(K,\mathcal{F})$-contrivance $(x,y,\alpha,h,\mathcal{Q})$ in $G_1$, 
where $\mathcal{Q}$ has cost at most 
$6\ell-2$. When the algorithm checks this ordered frame and this quintuple, the algorithm of Lemma \ref{cleanprism} 
outputs a set $X$ that contains all $K$-major vertices and does not intersect $V(K)$; so $K$ is clean in $G_1\setminus X$. The
algorithm of Theorem
\ref{thm:detectcleanprism}, applied to $G_1\setminus X$ cannot output that there is no long near-prism
that is the lightest among all long near-prisms, and has ordered frame $\mathcal{F}$, and is clean, because there is one.
Thus it will output that $G_1$ contains a long near-prism. This proves correctness, and so proves Theorem \ref{alg:longprisms2}.

\end{proof}

\section{Detecting a clean lightest long even hole}

Let us say a graph $G$ is a {\em candidate}
if it contains no long even hole of length at most $2\ell$, no long jewel of order at most $\ell+1$, no long theta,
no long ban-the-bomb, and no long near-prism. Thus, candidates are prospects.

Let $C$ be a hole in a graph $G$. We recall that a vertex $x\in V(G)\setminus V(C)$ is {\em $C$-major} if no three-vertex
path of $C$ contains
all the neighbours of $x$ in $V(C)$, and $C$ is {\em clean}
if there is no $C$-major vertex.
In this section we provide an algorithm to detect a clean lightest long even hole in a candidate if there is one.
We begin with:

\begin{lem}\label{triad}
Let $G$ be a candidate, and let 
$C$ be a shortest long even hole in $G$, and let $x$ be $C$-major. Then $x$ has three pairwise nonadjacent neighbours 
in $V(C)$, and for every three-vertex path $Q$ of $C$, $x$ has at least two 
neighbours  in $V(C)\setminus V(Q)$.
\end{lem}

\begin{proof} Since $G$ is a candidate it follows that $C$ has length at least $2\ell+2$.
If $x$ has at least five neighbours in $V(C)$ then both claims are true, so we assume that $x$ has at most four neighbours in $V(C)$,
say $v_1\ll v_k$ in order, where $2\le k\le 4$. If $k=2$ let $P_1,P_2$ be the two 
$v_1v_2$-paths of $C$, and if $k\in \{3,4\}$ let $P_i$ be the $v_iv_{i+1}$-path of $C$ not containing $v_{i+2}$ for $1\le i\le k$
(reading subscripts modulo $k$). 
\\
\\
(1) {\em For $1\le i\le k$, if $P_i$ has length at least $\ell-2$ then $P_i$ is odd, and the path $C\setminus P_i^*$ has length at least 
$\ell+2$.}
\\
\\
If $P_i$ has length at least $\ell-2$, then (reading subscripts modulo $k$) the hole $x\dd v_i\dd P_i\dd v_{i+1}\dd x$ is long 
and shorter than $C$, and therefore odd, and so $P_i$ is odd. Consequently $C\setminus P_i^*$ is also odd, since $C$ is even; and hence the paths 
$C\setminus P_i^*$, $v_1\dd x\dd v_2$ and $v_1\dd P_i\dd v_2$ form a long jewel, which therefore has order at
least $\ell+2$, that is, $C\setminus P_i^*$ has length at least  
$\ell+2$.
This proves (1).

\bigskip

Let $P_1$ be the longest of $P_1\ll P_k$.
If $k=2$, then $P_1$ is long, and so (1) implies that the paths $P_1$, $P_2$ and $v_1\dd x\dd v_2$ form a long theta, a contradiction, 
so $k\ge 3$. Suppose that $k=3$. If $P_2,P_3$ both have length at least three then both claims are true, so we may assume that
$P_2$ has length at most two. So $P_1$ is long, and hence so is $P_3$, by (1), and therefore they are both odd, by (1) again.
Thus $P_2$ is even, and so has length two,
and hence $G[V(C)\cup \{x\}]$
is a long ban-the-bomb, a contradiction. This proves that $k=4$.

If two of $P_2,P_3,P_4$ have length at least two, then both claims are true; so we may assume that $P_2$ has length one, and one
of $P_3,P_4$ has length one, and therefore $P_1$ is long. Now there are two cases. If $P_3$ has length one then 
$P_4$ is long, by (1), and so $G[V(C)\cup \{x\}]$
is a long ban-the-bomb, a contradiction; and if $P_4$ has length one then $P_3$ is long, by (1), and  $G[V(C)\cup \{x\}]$
is a long near-prism, a contradiction. This proves Lemma \ref{triad}.

\end{proof}

Let $C$ be a shortest long even hole. For $u,v$ distinct and non-adjacent vertices in $V(C)$ we call a
$uv$-path $Q$ a {\em shortcut} if $V(Q)$ contains no $C$-major vertices and $Q$ has length less than $d_C(u,v)$.
We begin by proving the following.

\begin{theorem} \label{thm:shortcuts}
Let $G$ be a candidate and let $C$ be a shortest long even hole in $G$.
Then $C$ has no shortcut.
\end{theorem}

\begin{proof}
Suppose that $G$ has a shortest long even hole $C$ with a shortcut $Q$. Thus $|E(C)|\ge 2\ell+2$, since $G$ is a candidate.
Choose $C,Q$ to minimize $|E(Q)|$, and subject to that, to maximize $d_C(u,v)$, where $u,v$ are the ends of $Q$. 
It follows that $Q^*\cap V(C)=\emptyset$. 
Let $Q$ have vertices $u \dd q_1 \dd q_2 \cc q_k \dd v$ in order. It follows that $Q$ has length $k+1$, and so 
$d_C(u,v)\ge k+2$. Consequently $k>1$, since $Q$ contains no $C$-major vertices.
\\
\\
(1) {\em The set of neighbours of $q_1$ in $V(C)$ is a clique, and the same holds for $q_k$, and $q_1,q_k$ have no common 
neighbour in $V(C)$.}
\\
\\
Suppose that $q_1$ has two nonadjacent neighbours in $V(C)$, say $x,y$. Since $q_1$ is not $C$-major, there is a vertex $z$ of $C$
such that $x\dd z\dd y$ is a path of $C$, and every neighbour of $q_1$ in $V(C)$ is one of $x,y,z$. Let $C'$ be the hole induced on
$(V(C)\setminus \{z\})\cup \{q_1\}$. Then $C'$ has the same length as $C$, and so is a shortest even hole, and
$d_{C'}(q_1,v)=d_C(z,v)\ge d_C(u,v)-1$. Let $Q'= Q\setminus \{u\}$. From the choice of $C,Q$ it follows that $Q'$ is not a shortcut for $C'$, and 
so some vertex of $Q'$ is $C'$-major, and hence is adjacent to $q_1$. Consequently $q_2$ is $C'$-major, and yet all its neighbours
in $V(C')$ except $q_1$ lie in a three-vertex path of $C$ and hence of $C'$, contrary to Lemma \ref{triad}.
This proves the first 
assertion of (1). The second follows since $d_C(u,v)\ge k+2\ge 4$. This proves (1).
\\
\\
(2) {\em If $1\le i\le k$ and $q_i$ is adjacent to $w\in V(C)\setminus \{u,v\}$, then $d_C(u,w)=|E(R)|$, where $R$ is the $uw$-path
of $C\setminus \{v\}$. The same holds with $u,v$ exchanged.}
\\
\\
Suppose not; then the shorter of the two $uw$-paths of $C$ strictly includes one of the $uv$-paths of $C$, and so has length more than $d_C(u,v)$,
contradicting the choice of $C,Q$. This proves (2).
\\
\\
(3) {\em One of $q_2 \ll q_{k-1}$ has a neighbour in $V(C)$.}
\\
\\
Suppose not. By (1), $q_1$ either has one, or two adjacent, neighbours in $V(C)$, and the same for $q_k$. There are two minimal
paths $R_1,R_2$ of $C$ with one end adjacent to $q_1$ and the other to $q_k$, and since the sum of their lengths is at least 
$|E(C)|-2\ge 2\ell$, we may assume that $R_2$ is long. 
Let the ends of $R_2$ be $u', v'$ where $u'$ is adjacent to $q_1$ and $v'$ to $q_k$. Let $S$ be the $u'v'$-path of $C$
different from $R_2$, and let $Q'$ be the path $u'\dd q_1\dd Q\dd q_k\dd v'$.
Now $Q'$ has the same length as $Q$, and therefore less than $d_C(u,v)\le |E(S)|$. Consequently the hole
$Q'\cup R_2$ has length less than $C$, and it is long and therefore odd. So $Q', S$ have different parity.
If $R_1$ is not long, then $S$ has length at most $\ell+1$, and so does $Q$, and hence the paths $S,Q', R_2$
form a long jewel of order at most $\ell+1$, a contradiction. So $R_1$ is long.

Not both $q_1,q_k$ have a unique neighbour in $V(C)$, since $G$ contains no long theta, and they do not both have two adjacent neighbours, since
$G$ contains no long near-prism. Thus we may assume that $q_1$ has two adjacent neighbours $x,v'$ in $V(C)$, and $q_k$ has exactly one 
(namely $v=v'$). Since $Q', S$ have different parity, it follows that the hole $x\dd q_1\dd Q\dd v\dd R_1\dd x$ is long, even,
and shorter than $C$, a contradiction.
This proves (3). 

\bigskip

Let $L_1$ be a $uv$-path of $C$ such that one of $q_2 \ll q_{k-1}$ has a neighbour in $V(L_1)$, and let $L_2$ be the other $uv$-path
of $C$. 
\\
\\
(4) {\em $Q^*$ is anticomplete to $L_2^*$.}
\\
\\
Choose $i\in \{2\ll k-1\}$ such that $q_i$ has a neighbour $w_1$ in $V(L_1)$, and suppose that there exists $j\in \{1\ll k\}$
such that $q_j$ has a neighbour $w_2$ in $V(L_2)$. By exchanging $u,v$ if necessary, we may assume that $i\le j$.
See figure \ref{fig:step4}.

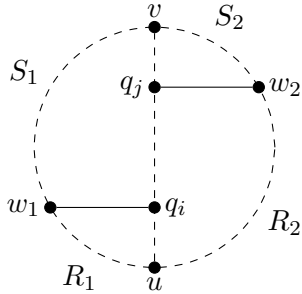
\begin{figure}[H]
\centering

\begin{tikzpicture}[scale=0.8,auto=left]
\tikzstyle{every node}=[inner sep=1.5pt, fill=black,circle,draw]

\def\r{2}
\def\s{2.5}
\draw[domain=0:360,smooth,variable=\x,dashed] plot ({\r*cos(\x)},{\r*sin(\x)});
\node (u) at ({\r*cos(270)},{\r*sin(270)}) {};
\node (v) at ({\r*cos(90)},{\r*sin(90)}) {};
\draw[dashed] (u)--(v);
\node (w1) at ({\r*cos(210)},{\r*sin(210)}) {};
\node (w2) at ({\r*cos(30)},{\r*sin(30)}) {};
\node (qi) at (0,{\r*sin(210)}) {};
\node (qj) at (0,{\r*sin(30)}) {};

\draw (qi)--(w1);
\draw (qj)--(w2);

\tikzstyle{every node}=[]
\draw (u) node [below]           {$u$};
\draw (v) node [above]           {$v$};
\draw (w1) node [left]           {$w_1$};
\draw (w2) node [right]           {$w_2$};
\draw (qi) node [right]           {$q_i$};
\draw (qj) node [left]           {$q_j$};
\node at ({\s*cos(240)}, {\s*sin(240)}) {$R_1$};
\node at ({\s*cos(150)}, {\s*sin(150)}) {$S_1$};
\node at ({\s*cos(330)}, {\s*sin(330)}) {$R_2$};
\node at ({\s*cos(60)}, {\s*sin(60)}) {$S_2$};

\end{tikzpicture}
\caption{For step (4).} \label{fig:step4}
\end{figure}

From the choice of $C,Q$ it follows that $w_1,w_2\ne u,v$.
For $i = 1,2$, let $R_i$ be the $uw_i$-path of $C\setminus \{v\}$, and let $S_i$ be the $vw_i$-path of $C\setminus \{u\}$.
Let $R_i, S_i$ have length $r_i, s_i$ for $i = 1,2$.
Each of the paths $u\dd q_1\dd q_i\dd w_1$, $w_1\dd q_i\cc q_j\dd w_2$, $w_2\dd q_j\dd q_k\dd v$ is strictly shorter than $Q$
(because $2\le i\le k-1$), and hence none of them is a shortcut. From (2), it follows that $r_1=d_C(u,w_1)\le i+1$, and 
$d_C(w_1,w_2)\le j-i+2$, and $s_2\le k-j+2$. Adding, we deduce that
$$r_1+d_C(w_1,w_2)+s_2\le (i+1)+(j-i+2)+(k-j+2)=k+5=|E(Q)|+4.$$

But $d_C(w_1,w_2)=\min(r_1+r_2,s_1+s_2)$. Suppose that $d_C(w_1,w_2)=r_1+r_2$. It follows that $r_1+(r_1+r_2)+s_2\le |E(Q)|+4$,
but $r_2+s_2>|E(Q)|$ since $Q$ is a shortcut, and so $r_1\le 1$, and therefore $d_C(w_1,v)\ge d_C(u,v)-1$. But $i>1$, and so 
$w_1\dd q_i\cc q_k\dd v$ is a shortcut for $C$, contradicting the choice of $C,Q$. Thus $d_C(w_1,w_2)=s_1+s_2< r_1+r_2$.

Hence $r_1 + (s_1+s_2) + s_2\le |E(Q)|+4$. But $r_1+s_1\ge d_C(u,v)>|E(Q)|$, and so $s_2=1$. Since $u\dd Q\dd q_j\dd w_2$
is not a shortcut for $C$ that is shorter than $Q$, it follows that $j=k$. Since $r_1\le i+1$ and 
$$s_1+1=d_C(w_1,w_2)\le j-i+2=k-i+2,$$
we deduce (adding) that $r_1+s_1\le k+2$. But $r_1+s_1>|E(Q)|=k+1$, and so equality holds; that is, $r_1=i+1$
and $s_1=k-i+1$, and $r_1+s_1=|E(Q)|+1$. Hence $r_1+s_1\le r_2+s_2$; and since $d_C(u,w_2)\le d_C(u,v)$ (from the choice
of $C,Q$, since otherwise $u\dd Q\dd q_k\dd w_2$ would be a shortcut for $C$ contrary to the choice of $u,v$), it follows 
that $r_2+s_2=r_1+s_1=|E(C)|/2$. Moreover,
we showed that $q_j=q_k$ and $w_2$ is adjacent to $v$, on the assumption that $i\le j$; and it follows from the symmetry
that the only edges between $Q^*$ and $L_2^*$ are the edge $q_kw_2$ and possibly an edge from $q_1$ to the neighbour ($w_3$) say
of $u$ in $L_2$, say $w_3$. If the latter edge does not exist, then $u\dd Q\dd q_k\dd w_2\dd L_2\dd u$ is an even hole,
of length $|E(C)|-2$, a contradiction; so $q_1$ is adjacent to $w_3$. We already showed that $d_C(w_1,w_2)=s_1+1$, and 
it follows by the same argument with $u,v$ exchanged that
$d_C(w_1,w_3)=r_1+1=i+2$, and so the path $w_3\dd q_1\cc q_i\dd w_1$ is a shortcut for $C$, a contradiction. This proves (4).
\\
\\
(5) {\em $|E(L_1)|=|E(Q)|+1\le |E(L_2)|$.}
\\
\\
Choose $i\in \{2\ll k-1\}$ such that $q_i$ has a neighbour $w\in L_1^*$. Since $u\dd Q\dd q_i\dd w$ and $w\dd q_i\cc Q\dd v$ 
are not shortcuts for $C$, it follows from (2) that the sum of their lengths is at least $|E(L_1)|$, and so $|E(L_1)|\le |E(Q)|+2$. Since one of $L_1,L_2$
is long (because $|E(C)|\ge 2\ell$), it follows that the hole $Q\cup L_2$ is long, and shorter than $C$, and therefore odd;
and so $Q,L_2$ have opposite parity. Since $L_1,L_2$ have the same parity, and $|E(L_1)|>|E(Q)|\ge |E(L_1)|-2$, we deduce that
$|E(L_1)|=|E(Q)|+1$. Since $|E(L_2)|>|E(Q)|=|E(L_1)|-1$ it follows that $|E(L_1)|\le |E(L_2)|$. This proves (5).

\bigskip

By (5), we may number the vertices of $L_1$ as $u\dd c_1\cc c_{k+1}\dd v$ in order. 
\\
\\
(6) {\em For $1\le i\le k$, if $q_i$ is adjacent to $c_j$ where $1\le j\le k+1$, then $j\in \{i,i+1\}$.}
\\
\\
If $i\in \{1,k\}$ this is true since $q_1,q_k$ are not $C$-major, so we may assume that $2\le i\le k-1$. The path $u\dd Q\dd q_i\dd c_j$ has length $i+1$, 
shorter than $Q$, and so is not a shortcut; and hence by (2), $j=d_C(u,c_j)\le i+1$. Since
$c_j\dd q_i\dd Q\dd v$ is not a shortcut, it follows that $k+2-j=d_C(v,c_j)\le k+2-i$, and so $i\le j$. This proves (6).

\bigskip
By (3), (4) and (6), there exists $i\in \{2\ll k-1\}$ such that $q_i$ is adjacent to one of $c_i, c_{i+1}$, and by exchanging $u,v$
if necessary, we may assume that $q_i$ is adjacent to $c_i$. By (6), 
$$u\dd c_1\cc c_i\dd q_i\dd Q\dd v\dd L_2\dd u$$
is a hole $C'$ say. Since the paths $c_i\dd c_{i+1}\cc c_{k+1}\dd v$
and $c_i\dd q_i\cc q_k\dd v$ have the same length, it follows that $C'$ has the same length as $C$, and so is a shortest long 
even hole. From the choice of $C,Q$, the path $u\dd q_1\cc q_i$ is not a shortcut for $C'$. But its length is $i<d_{C'}(u,q_i)$,
and so one of its vertices is $C'$-major. Hence there exists $h\in \{1\ll i-1\}$ such that $q_h$ is $C'$-major and not
$C$-major, and so $q_h$ has a neighbour in $\{q_i\ll q_k\}$. But $q_h$ is nonadjacent to $\{q_{i+1}\ll q_k\}$, 
and therefore $h=i-1$, so $q_{i-1}$ is $C'$-major. By Lemma \ref{triad}, at least two neighbours in $q_{i-1}$ in $V(C)$ are not in 
$\{c_{i-1}, c_i, q_i\}$, contrary to (4) and (6).
This proves Theorem \ref{thm:shortcuts}.

\end{proof}

We will also need:

\begin{theorem}\label{thm:easyreroute}
Let $C$ be a clean shortest long even hole in a candidate $G$.
Let $u,v$ be distinct, non-adjacent vertices in $V(C)$ with $d_C(u,v)\le |E(C)|/2-2$, and let $L_1$, $L_2$ be the two $uv$-paths of $C$ 
where $|E(L_1)| \le  |E(C)|/2-2$. Then
$P \cup L_2$ is a shortest long even hole for every shortest $uv$-path $P$ in $G$.
\end{theorem}

\begin{proof}
Let $P$ be a shortest $uv$-path in $G$, with vertices $u\dd p_1\cc p_k\dd v$. Since
$C$ is clean, it follows from Theorem \ref{thm:shortcuts} that $P$ has the same length as $L_1$.
Suppose that for some $i\in \{1\ll k\}$,
$p_i$ is equal or adjacent to some $w\in L_2^*$. By Theorem \ref{thm:shortcuts}, 
the path $u\dd p_1\cc p_i\dd w$ (or 
$u\dd p_1\cc p_i$ if $w=p_i$) is not a shortcut for $C$,
and so $i+1\ge d_C(u,w)$. Since $i+1\le k+1=|E(L_1)|$ it follows that the shorter $uw$-path of $C$ is a subpath of $L_2$,
and hence $i+1\ge d_C(u,w)=d_{L_2}(u,w)$. Similarly $k-i+2\ge d_{L_2}(w,v)$. Consequently 
$$|E(P)|+2=k+3\ge d_{L_2}(u,w)+d_{L_2}(w,v)\ge d_{L_2}(u,v)=|E(L_2)|\ge |E(L_1)|+4,$$
a contradiction. This proves Theorem \ref{thm:easyreroute}.

\end{proof}

This can be strengthened: it is shown in~\cite{lindathesis} that 
\begin{theorem}\label{thm:cleanshortest}
Let $C$ be a clean shortest long even hole in a candidate $G$.
Let $u,v$ be distinct, non-adjacent vertices in $V(C)$, and let $L_1$, $L_2$ be the two $uv$-paths of $C$ 
where $|E(L_1)| \leq |E(L_2)|$. Then
for every shortest $uv$-path $P$ in $G$, either $P \cup L_2$ is a clean shortest long even hole in $G$, or $|E(L_1)| = |E(L_2)|$ 
and $P \cup L_1$ is a clean shortest long even hole in $G$.
\end{theorem}
We will not need this stronger form, however, so we omit it here.

Let us fix a linear order of the edges of $G$; then we can search for a lightest long even hole, instead of just a shortest one, and
it is easier to find if it exists. For instance, from Theorem \ref{thm:easyreroute} we obtain
\begin{theorem}\label{thm:lexeasyreroute}
Let $C$ be a lightest long even hole in a candidate $G$.
Let $u,v$ be distinct, non-adjacent vertices in $V(C)$ with $d_C(u,v)\le |E(C)|/2-2$, and let $L_1$, $L_2$ be the two $uv$-paths of $C$
where $|E(L_1)| \le  |E(C)|/2-2$. Then $L_1$ is the lightest $uv$-path in $G$ that contains no $C$-major vertices.
\end{theorem}

\begin{proof}
Let $P$ be the lightest $uv$-path in $G$ that contains no $C$-major vertices, and let $G'$ be the graph obtained from $G$
by deleting all $C$-major vertices. Thus $P$ is the lightest $uv$-path in $G'$. But $C$ is clean in $G'$,
and so by Theorem \ref{thm:easyreroute},
$P \cup L_2$ is a shortest long even hole. It cannot be lighter than $C$, and so $P$ is not lighter than $L_1$. 
On the other hand
$L_1$ is not lighter than $P$, since $P$ is the lightest $uv$-path in $G'$. Hence $P=L_1$. This proves Theorem \ref{thm:lexeasyreroute}.

\end{proof}

Now the main result of the section:

\begin{theorem}\label{thm:detectingCleanSLEH}
There is an algorithm with the following specifications:
\begin{description}
\item[Input:] A candidate $G$, and a linear ordering of $E(G)$.
\item[Output:] Decides either that $G$ has a long even hole  or that there is no clean lightest long even hole in $G$.
\item[Running time:] $\mathcal{O}(|G|^4)$.
\end{description}
\end{theorem}

\begin{proof}
For all distinct $u,v\in V(G)$, compute  a lightest $uv$-path $Q(uv)=Q(vu)$, and compute the set $N(uv)$ of all vertices that belong to or have
a neighbour in $Q(uv)^*$. 
Enumerate all triples $(v_1, v_2, v_3)$ of distinct vertices in $G$, and check whether  
$$Q(v_1v_2)\cup Q(v_2v_3)\cup Q(v_3v_1)$$
is a long even hole, and if so, report this and stop. If all triples are examined without success, report that 
 $G$ contains no clean lightest long even hole. That concludes the description of the algorithm.

Each triple can be handled in time $\mathcal{O}(|G|)$ (by using the sets $N(uv)$), and so the total
running time is $\mathcal{O}(|G|^4)$. 

To prove correctness, let $C$ be a clean lightest long even hole in $G$; we must show that there is a triple $(v_1,v_2,v_3)$
for which the algorithm will find a long even hole. Since  $C$ has length at least $12$ and hence $|E(C)|\le 3(|E(C)|/2-2)$, there exist
$v_1,v_2,v_3 \in V(C)$ such that each pair of vertices in this triple is 
joined by a path of $C$ of length at most $|E(C)|/2-2$ that does not contain the third vertex in the triple.
By Theorem \ref{thm:lexeasyreroute} $Q(v_1v_2), Q(v_2v_3)$ and $Q(v_3v_1)$ are all paths of $C$ and they have union $C$.
This proves \ref{thm:detectingCleanSLEH}.

\end{proof}

\section{Cleaning a shortest long even hole}
Our method of cleaning is very much like that used for shortest long near-prisms,
and the next result is an analogue of Lemma \ref{distant}.
Let $C$ be a shortest long even hole in a candidate $G$.
For a $C$-major vertex $x$, we call a path $P$ of $C$ of length at least two a {\em $(C,x)$-gap} if both ends of $P$ are 
neighbours of $x$ and no interior vertex of $P$ is adjacent to $x$. Thus, adding $x$ to $P$ yields a hole.

We begin with:

\begin{lem}\label{twoingap}
Let $C$ be a shortest long even hole in $G$, and let $x,y$ be nonadjacent $C$-major vertices. Let $P$ be a $(C,x)$-gap of 
length at least $\ell-2$, with ends 
$p_1,p_2$. If $y$ has a neighbour in $V(P)$, then either
\begin{itemize}
\item for some $i\in \{1,2\}$, some neighbour $v$ of $y$ in $V(P)$ satisfies $d_P(p_i,v)\le \ell-5$;
or 
\item for some $i\in \{1,2\}$, $y$ is adjacent to a neighbour of $p_i$ in $C$; or
\item $y$ has exactly two neighbours in $V(P)$ and they are adjacent.
\end{itemize}
\end{lem}

\begin{proof}
Let $Q$ be the $p_1p_2$-path of $C$ different from $P$. Thus $Q$ has length at least three.
The hole $x\dd p_1\dd P\dd p_2\dd x$ is long and
shorter than $C$, and so odd, and hence $P,Q$ are odd. 
Let $R$ be the graph obtained from $Q$ by deleting its first two and last two vertices. 
We may assume that the first two bullets of the theorem are false.
\\
\\
(1) {\em $x$ and $y$ have a neighbour in $V(R)$.}
\\
\\
By Lemma \ref{triad},
$x$ has a neighbour in $V(R)$.
Suppose that $y$ does not. Since the first two bullets of the theorem are false, 
it follows that all neighbours of
$y$ in $V(C)$ belong to $P^*$, and hence $y$ has two nonadjacent neighbours in $V(P)$. 
Let $P'$ be the induced $p_1p_2$-path with interior in $V(P)\cup \{y\}$ that contains $y$.
Since the first bullet of the theorem is false, it follows that $P'$ has length at least $2(\ell-4)+2\ge \ell$,
and so the hole $x\dd p_1\dd P'\dd p_2\dd x$ is long and shorter than $C$, and so odd. Hence $P'$ is 
odd; but $Q$ is also odd, and $P'\cup Q$ is a long even hole shorter than $C$, a contradiction.
This proves (1).

\bigskip

By (1), there is an induced $xy$-path with interior in $V(R)$, say $M$.
By hypothesis, $y$ has at least one neighbour in $V(P)$. If $y$ has only one neighbour
$v$ in $V(P)$, then there is a long theta formed by the two $xv$-paths with interior in $V(P)$ and $x\dd M\dd y\dd v$,
(because the first two paths both have length at least $\ell-3$, and the third has length at least three)
a contradiction. If $y$ has two nonadjacent neighbours in $V(P)$, there is a long theta formed by the two
induced $xy$-paths with
interior in $V(P)$ and $M$,
again a contradiction. Hence $y$ has exactly two neighbours in $V(P)$ and they are adjacent.
This proves Lemma \ref{twoingap}.

\end{proof}

Let $C$ be a shortest long even hole. A {\em $C$-contrivance} is a six-tuple $(x,y,p_1,p_2,m,\mathcal{Q})$, where
\begin{itemize}
\item $x,y$ are $C$-major vertices (possibly $y=x$), and there is a $(C,x)$-gap $P$ with ends $p_1,p_2$ and midpoint $m$
such that every $C$-major vertex has a neighbour in $V(P)$;
\item $\mathcal{Q}$ is a set of paths of $C$, pairwise anticomplete;
\item every neighbour of $x$ or $y$ in $V(P)$ belongs to $\mathcal{Q}^*$; and 
\item $x,y$ and all $C$-major vertices nonadjacent to both $x,y$ have a neighbour in $\mathcal{Q}^*$.
\end{itemize}
Its {\em cost} is the number of vertices in $V(\mathcal{Q})$.

These objects will be the analogue of $(K,\mathcal{F})$-contrivances, and we will use them in the same way.
The next result is an analogue of Lemma \ref{contrivance}.

\begin{lem} \label{lem:Ccontrivance}
Let $G$ be a candidate and let $C$ be a shortest long even hole in $G$ such that for some $C$-major vertex $x$, there is 
a $(C,x)$-gap.
Then there is a $C$-contrivance with cost at most $4\ell-4$.
\end{lem}

\begin{proof}
Choose a maximal
path $P$ of $C$ such that there is a $C$-major vertex $x$ for which $P$ is a $(C,x)$-gap. Let $P$ have ends $p_1,p_2$, and 
let $m$ be a midpoint of $P$. 
It follows that every $C$-major vertex has a neighbour in $V(P)$.
For $i \in \{1,2 \}$ let $Q_i$ be the path of $C$ whose vertex set consists of all vertices of $V(P)$ with $P$-distance
at most $\ell-4$ from $p_i$ and the two vertices of $V(C)\setminus V(P)$ with $C$-distance at most two from $p_i$.

Let $S$ be the set of all $C$-major vertices with no neighbour in $Q_1^* \cup Q_2^* \cup \{x\}$. 
We may assume that $S \neq \emptyset$, because otherwise  $(x,x,p_1,p_2,m, \mathcal{Q})$ is a
$C$-contrivance satisfying the theorem, where $Q$ is the set of components of $G[V(Q_1\cup Q_2)]$.
Hence $Q_1,Q_2$ are vertex-disjoint.
For each $y \in S$, we define $P_y$ to be the $(C,y)$-gap with $p_1 \in P_y^*$.
Choose $y\in S$ with $|E(P_{y}) \setminus E(P)|$ maximum, and let $p_3, p_4$ be the ends of $P_{y}$, where $p_3\notin V(P)$. (By Lemma \ref{twoingap}, one end of $P_y$ is not in $V(P)$.)
Since $y$ has no neighbour in $Q_1^* \cup Q_2^* \cup \{x \}$ and $y$ has at least one
neighbour in $V(P)$, it follows from Lemma \ref{twoingap} that $y$ has exactly two neighbours in $V(P)$ and
they are adjacent. One of them is $p_4$; let the other be $p_5$. 

Let $R$ denote the path $p_1\dd P_{y}\dd p_3$.
For $i = 3,4$, let $Q_i$ be the path of $C$ whose vertex set consists of
all vertices of $V(P_{y})$ with $P_{y}$-distance at most $\ell-4$ from $p_i$ and the two vertices of
$V(C) \setminus V(P_{y})$ with $C$-distance at most two from $p_i$.
Then $p_5 \in Q_4^*$.
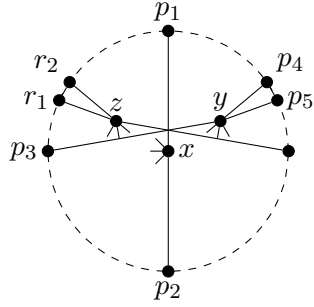
\begin{figure}[H]
\centering

\begin{tikzpicture}[scale=0.8,auto=left]
\tikzstyle{every node}=[inner sep=1.5pt, fill=black,circle,draw]

\def\w{3.5}
\def\r{2}
\def\s{1}
\def\t{.3}
\draw[domain=0:360,smooth,variable=\x,dashed] plot ({\r*cos(\x)},{\r*sin(\x)});
\node (x) at (0,0) {};
\node (p1) at ({\r*cos(90)},{\r*sin(90)}) {};
\node (p2) at ({\r*cos(270)},{\r*sin(270)}) {};
\node (y) at ({\s*cos(30)},{\s*sin(30)}) {};
\node (z) at ({\s*cos(150)},{\s*sin(150)}) {};
\node (p4) at ({\r*cos(35)},{\r*sin(35)}) {};
\node (p5) at ({\r*cos(25)},{\r*sin(25)}) {};
\node (r1) at ({\r*cos(155)},{\r*sin(155)}) {};
\node (r2) at ({\r*cos(145)},{\r*sin(145)}) {};
\node (p3) at ({\r*cos(180)},{\r*sin(180)}) {};
\node (p6) at ({\r*cos(0)},{\r*sin(0)}) {};

\draw (x) -- ({\t*cos(135)},{\t*sin(135)});
\draw (x) -- ({\t*cos(180)},{\t*sin(180)});
\draw (x) -- ({\t*cos(225)},{\t*sin(225)});
\draw (y) -- ({\s*cos(30)+\t*cos(220)},{\s*sin(30)+\t*sin(220)});
\draw (y) -- ({\s*cos(30)+\t*cos(260)},{\s*sin(30)+\t*sin(260)});
\draw (y) -- ({\s*cos(30)+\t*cos(300)},{\s*sin(30)+\t*sin(300)});
\draw (z) -- ({\s*cos(150)+\t*cos(320)},{\s*sin(150)+\t*sin(320)});
\draw (z) -- ({\s*cos(150)+\t*cos(280)},{\s*sin(150)+\t*sin(280)});
\draw (z) -- ({\s*cos(150)+\t*cos(240)},{\s*sin(150)+\t*sin(240)});

\foreach \to/\from in {y/p3,y/p4,y/p5,z/r1,z/r2,z/p6,x/p1,x/p2,p4/p5,r1/r2}
\draw [-] (\from) -- (\to);

\tikzstyle{every node}=[]
\draw (x) node [right]           {$x$};
\draw (y) node [above]           {$y$};
\draw (z) node [above]           {$z$};
\draw (p1) node [above]           {$p_1$};
\draw (p2) node [below]           {$p_2$};
\draw (p3) node [left]           {$p_3$};
\draw (p4) node [above right]           {$p_4$};
\draw (p5) node [right]           {$p_5$};
\draw (r1) node [left]           {$r_1$};
\draw (r2) node [above left]           {$r_2$};

\end{tikzpicture}
\caption{For Lemma \ref{lem:Ccontrivance}.} \label{fig:Ccontrivance}
\end{figure}
\noindent(1) {\em Every $C$-major vertex has a neighbour in $Q_1^* \cup Q_2^*\cup Q_3^*\cup Q_4^* \cup \{x,y \}$.}
\\
\\
Suppose that $z$ is $C$-major and has no neighbour in this set. Thus $z\ne x,y$.
Since $z$ has a neighbour in $V(R)$ from the choice of $y$, it follows from Lemma \ref{twoingap} applied to $z$ and $P_{y}$
that $z$ has exactly two neighbours in $V(P_{y})$ and they are adjacent, say $r_1, r_2$.
Since $z$
has a neighbour in $V(R)$ and $z$ is not adjacent to $p_1$, it follows that $r_1, r_2 \in V(R)$.
Number them such that $p_3,r_1,r_2,p_1$ are in order in $P_y$.
Since $z$ has a neighbour in $V(P)$, and no neighbour in $p_1\dd P\dd p_4$, there is a $zp_5$-path $M$ with interior in the vertex set
of $p_5\dd P\dd p_2$. But then there is a long prism with bases $\{y,p_4,p_5\}$, $\{z,r_1,r_2\}$, and constituent paths
$M$, 
$y\dd p_3\dd P_y\dd r_1$ and $r_2\dd P_y\dd p_4$, a contradiction.
This proves (1).

\bigskip

Let $\mathcal{Q}$ be the set of components of the subgraph induced on $V(Q_1)\cupcup V(Q_4)$.
From (1), it follows that $(x,y, p_1,p_2,m,\mathcal{Q})$ satisfies the theorem. This proves Lemma \ref{lem:Ccontrivance}.
\end{proof}

If we know a $C$-contrivance  $(x,y,p_1,p_2,m,\mathcal{Q})$ for some lightest long hole $C$ (but we do not know $C$),
it is possible to construct a set $X$ of vertices that contains
all $C$-major vertices and does not intersect $C$. To do so, we first need to reconstruct the path $P$ (in the notation above).
If we could do that, then since every $C$-major vertex has a neighbour in one of
$P^*$, $\mathcal{Q}^*$, and no vertex in $V(C)\setminus (V(P)\cup V(\mathcal{Q}))$ has such a neighbour, we would have the
desired set $X$. So, how to reconstruct $P$? As for long near-prisms, it is easier if $C$ is the lightest long even
hole, rather than just the shortest, and then we would like to use Theorem \ref{thm:lexeasyreroute} as the analogue of Lemma \ref{prismjump}. There is a slight
problem that did not arise for long near-prisms: the path $P$ we are trying to reconstruct
might have length more than $|E(C)|/2$ or close to that, and then
we cannot use Theorem \ref{thm:lexeasyreroute} directly. But if we know a midpoint $m$ of $P$, then $m$ divides
$P$ into two subpaths that are short enough to be reconstructed via Theorem \ref{thm:lexeasyreroute}. For that reason
we put the extra vertex $m$ in the definition of a $C$-contrivance.
We can now prove the main result of this section, an analogue of Lemmas \ref{firstpath}, \ref{reconstruct} and \ref{cleanprism}.

\begin{theorem} \label{alg:cleaningSLEH}
There is an algorithm with the following specifications:
\begin{description}
\item[Input:] A candidate $G$, and a linear ordering of $E(G)$.
\item[Output:] A list of $\mathcal{O}(|G|^{4\ell -1})$ sets with the following property: for every lightest long 
even hole $C$ there is some $X$ in the list such that $X$ contains all $C$-major vertices and $X \cap V(C) = \emptyset$.
\item[Running time:] $\mathcal{O}(|G|^{4\ell + 2})$
\end{description}
\end{theorem}
\begin{proof}
First we output the set of all neighbours of $y$ different from $x,z$, for every
induced path $x\dd y\dd z$ in $G$. 

Now guess three vertices $x,y,m$ of $G$ and a set $\mathcal{Q}$ of induced paths of $G$, pairwise 
anticomplete, with cost at most $4\ell-4$; and guess $p_1,p_2 \in V(\mathcal{Q})$. If one of $x,y$ belongs to 
$V(\mathcal{Q})$
or has no neighbour in $\mathcal{Q}^*$, go on to the next guess.

Let $Z_1$ be the set of vertices in $V(G) \setminus V(\mathcal{Q})$ with a neighbour in $\mathcal{Q}^*$. Let $Z_2$
be the set of all vertices in $V(G)\setminus (V(\mathcal{Q})\cup \{m\})$ with a neighbour in $\{x,y\}$.
Let $G'=G \setminus (Z_1\cup Z_2)$, and let 
$R,S$ be the lightest $p_1m$-path and $p_2m$-path in
$G'$ respectively. (If these do not exist, or if
$R\cup S$ is not an induced path, go on to the next guess.)
Let $Z_3$ be the set of vertices in
$V(G) \setminus \mathcal{Q}^*$ with a neighbour in $\{x,y\}$ and a neighbour in the interior of $R\cup S$. 
Output $Z_1 \cup Z_3$, and go on to the next guess.
That completes the description of the algorithm.

There are $\mathcal{O}(|G|^{4\ell -1})$ guesses of $(x,y,p_1,p_2,m,\mathcal{Q})$ to check 
(because $p_1,p_2\in V(\mathcal{Q})$), and so the output list has size $\mathcal{O}(|G|^{4\ell -1})$. 
For each guess, we compute $Z_1,Z_2, Z_3$ in time $\mathcal{O}(|G|^3)$. Hence the 
total running time is $\mathcal{O}(|G|^{4\ell+2})$.

Now we prove the output is correct. 
Suppose that $C$ is a lightest long even hole in $G$.
If every $C$-major vertex is complete to $V(C)$, then the set $X$ satisfies our requirement, where 
$X$ is the set of all neighbours of $y$ different from $x,z$, for some three-vertex path $x\dd y\dd z$ of $C$.
So we may assume that some $C$-major vertex is not complete to $V(C)$.

By Lemma \ref{lem:Ccontrivance}, $G$ contains a $C$-contrivance $(x,y,p_1,p_2,m,\mathcal{Q})$ with cost at most
$4\ell-4$. We will show that when we guess this $C$-contrivance, we output the set $X$ that we need.
Let $P$ denote the $(C,x)$-gap with ends $p_1,p_2$ and with midpoint $m$. 
From the definition of $Z_1$, no vertex of $Z_1$ belongs to $V(C)$, since the paths in $\mathcal{Q}$ are paths of $C$.
It remains to show that $Z_3\cap V(C)=\emptyset$, and every $C$-major vertex belongs to $Z_1\cup Z_3$.

The path $C\setminus P^*$ contains all neighbours of $x$ in $V(C)$, and so by Lemma \ref{triad}, $C\setminus P^*$ has length at 
least four. Hence $|E(P)|\le |E(C)|-4$, and so
the paths $p_1\dd P\dd m$ and $m\dd P\dd p_2$ both have length
at most 
$\lceil |E(P)|/2\rceil \le |E(C)|/2-2$.
Moreover, $p_1\dd P\dd m$ is a path of $G'$, and so the algorithm will compute the lightest $p_1m$-path $R$ in $G'$,
since such a path exists. So $p_1\dd P\dd m$ is not lighter than $R$. But $R$ contains no $C$-major vertices of $G$, and
by Theorem \ref{thm:lexeasyreroute},
the path $p_1\dd P\dd m$ is the lightest $p_1m$-path of $G$ that contains no $C$-major vertices, so $R$ is not 
lighter than $p_1\dd P\dd m$. Consequently $R$ equals the path $p_1\dd P\dd m$. Similarly $S$ is the path 
$m\dd P\dd p_2$, and so $R\cup S=P$. Consequently $Z_3\cap V(C)=\emptyset$.

Now suppose that $z$ is a $C$-major vertex not in $Z_1$; we must show that $z\in Z_3$. Since every $C$-major vertex
that is nonadjacent to both $x,y$ has a neighbour in $\mathcal{Q}^*$, it follows that $z$ is adjacent to one of $x,y$.
Also $z$ has a neighbour in $V(P)$, since every $C$-major vertex has a neighbour in $V(P)$; and so $z$
has a neighbour in $P^*$. Thus $z\in Z_3$, as required.
This proves correctness, and so proves Theorem \ref{alg:cleaningSLEH}.

\end{proof}

\section{The main algorithm}

Now we prove our main result Theorem \ref{mainthm}, which we restate: 
\begin{theorem}\label{mainthm2}
For each even integer $\ell \geq 4$ there is an algorithm with the following specifications:
\begin{description}
\item[Input:] A graph $G$.
\item[Output:] Decides whether $G$ has an even hole of length at least $\ell$.
\item[Running time:] $\mathcal{O}(|G|^{9\ell+3})$.
\end{description}
\end{theorem}

\begin{proof}
The algorithm is as follows. At each step, if we find that $G$ contains a long even hole, we output that fact and stop, so
in steps 1,2,3,4,5,7 we can assume the algorithm called at that step outputs the negative answer.
Fix a linear ordering of $E(G)$.
\begin{enumerate}[\bf Step 1:]
\item Apply the algorithm of Theorem \ref{alg:shortlongevenholes} to test whether $G$ contains a long even hole of length 
at most $2\ell$ in time $\mathcal{O}(|G|^{2\ell})$. 
\item Apply the algorithm of Theorem \ref{alg:longjewels} to test 
whether $G$ contains a long jewel of order at most $\ell + 1$ in time $\mathcal{O}(|G|^{2\ell +1})$.
\item Apply the 
algorithm of Theorem \ref{alg:longtheta} to test whether $G$ contains a long theta in time $\mathcal{O}(|G|^{2\ell -1})$. 
\item Apply the algorithm of Theorem \ref{alg:longbanthebomb} to test whether $G$ contains a long ban-the-bomb, 
in time $\mathcal{O}(|G|^{2\ell + 1})$. (If we have not yet found a long even hole, then $G$ is a prospect.)
\item Apply the algorithm of 
Theorem \ref{alg:longprisms} to test whether $G$ contains a long near-prism, in time $\mathcal{O}(|G|^{9\ell+3})$. 
(If we have still not found a long even hole, then $G$ is a candidate.)
\item Apply the algorithm of Theorem \ref{alg:cleaningSLEH} to obtain a list $\mathcal{L}$ of subsets of $V(G)$
of length $\mathcal{O}(|G|^{4\ell-1})$ in time $\mathcal{O}(|G|^{4\ell+2})$,  with the property that for every lightest long even hole
$C$ of $G$ there exists $X\in \mathcal{L}$ with $X\cap V(C)=\emptyset$ that contains all $C$-major vertices.
\item For every $X\in \mathcal{L}$, apply
the algorithm of Theorem \ref{thm:detectingCleanSLEH} to $G\setminus X$, to decide that either $G\setminus X$ has a long even hole, 
or $G \setminus X$ has no clean lightest long even hole,
in time $\mathcal{O}(|G|^4)$ for each $X$, and so in time  $\mathcal{O}(|G|^{4\ell+3})$ altogether.
\item  Output that $G$ has no long even hole.
\end{enumerate}

For correctness, certainly if the algorithm returns that $G$ has a long even hole then that is true. For the converse,
suppose that $G$ has a long even hole, and hence a lightest long even hole $C$ say. Steps 1-5 will either output that there is a 
long even hole or decide that $G$ is a candidate, and we may assume the latter. Hence, with $\mathcal{L}$
is computed in step 6, there exists $X\in \mathcal{L}$ disjoint
from $V(C)$ and containing all $C$-major vertices. Then in step 7, since $C$ is a clean lightest long even hole of $G\setminus X$,
the algorithm of Theorem \ref{thm:detectingCleanSLEH} cannot report that  $G \setminus X$ has no clean lightest long even hole,
and so it will report that $G\setminus X$ has a long even hole, and we return this fact correctly.

For the running time, 
testing whether $G$ is a candidate (steps 1-5) takes time $\mathcal{O}(|G|^{9\ell+3})$, and determining whether the candidate $G$
contains a long even hole (steps 6-8) takes time $\mathcal{O}(|G|^{4\ell+3})$. Hence, the total running time is 
$\mathcal{O}(|G|^{9\ell+3})$. This proves Theorem \ref{mainthm2}.
\end{proof}



\end{document}